\documentclass[10pt,reqno]{amsart}
\usepackage{setspace}
\usepackage[foot]{amsaddr}
\usepackage[top=3cm,bottom=3cm,left=2.4cm,right=2.4cm]{geometry}
\usepackage[latin1]{inputenc}
\usepackage{amsmath,amsfonts,amssymb,amsthm,bm}
\usepackage{mathrsfs}
\usepackage{color}
\usepackage{enumitem}
\usepackage[english]{babel}
\usepackage{graphicx}
\usepackage{graphics}
\usepackage[T1]{fontenc}
\usepackage{float}
\usepackage{placeins}
\graphicspath{{figures/}}
\usepackage{mathpazo}
\usepackage{bbm}
\usepackage{xcolor}
\usepackage{pgfplots}
\usepackage{array}
\usepackage{indentfirst}
\usepackage{mathtools}
\usepackage{pdfpages}
\usepackage{multirow}
\usepackage{latexsym}
\usepackage{caption}
\usepackage{subcaption}
\usepackage{tikz}
\usepackage{tikz-cd}
\usetikzlibrary{%
	arrows.meta,%
    decorations.pathreplacing,%
    decorations.pathmorphing,%
    patterns,%
    intersections,%
    calc%
}
\usepackage{tcolorbox}
\tcbuselibrary{theorems}

\usepackage{adjustbox}
\usepackage{hyperref}
\usepackage{esint}

\definecolor{orange1}{rgb}{0.92900,0.69400,0.12500}%
\definecolor{purple1}{rgb}{0.49400,0.18400,0.55600}%
\definecolor{brown1}{rgb}{0.7,0.3400,0.12500}%

\definecolor{blue11}{rgb}{0.00000,0.44700,0.74100}%
\definecolor{red11}{rgb}{0.85000,0.32500,0.09800}%

\definecolor{color1}{rgb}{0.22,0.45,0.70}
\definecolor{color2}{rgb}{0.9, 0.17, 0.31}
\definecolor{color3}{rgb}{0.8, 0.25, 0.33}
\definecolor{color4}{RGB}{15, 5, 107}
\definecolor{blue2}{rgb}{0.2,0.2,0.7}
\definecolor{green2}{rgb}{0.13, 0.7, 0.67}
\definecolor{green3}{rgb}{0.3,0.6,0.4}
\definecolor{mygray}{gray}{0.40}
\definecolor{mylightgray}{gray}{0.70}

\pgfplotsset{compat=1.17}

\hypersetup{colorlinks=true,linkcolor=color3,citecolor=mygray,urlcolor=black,linktoc=page}

\makeatletter
\def\paragraph{\@startsection{paragraph}{4}%
  \z@\z@{-\fontdimen2\font}%
  {\normalfont\bfseries}}
\makeatother

\newtheorem{lemma}{Lemma}[section]
\newtheorem{proposition}{Proposition}[section]
\newtheorem{corollary}{Corollary}[section]

\newtheorem{remark}{Remark}[section]
\newtheorem{definition}{Definition}[section]

\newcommand{\E}{\mathbb{E}}
\newcommand{\F}{\mathcal{F}}

\newcommand{\Pro}{\mathbb{P}}
\newcommand{\R}{\mathbb{R}}
\newcommand{\C}{\mathbb{C}}
\newcommand{\dst}{\displaystyle}
\newcommand{\eps}{\varepsilon}

\newcommand{\Cov}{\textrm{Cov}}

\newcommand{\vsd}{\vspace{0.2cm}}

\newcommand{\re}[1]{\textcolor{color3}{#1}}

\newcommand{\dd }{\mathrm{d}}

\newcommand{\norme}[1]{\left\Vert #1\right\Vert}

\title{Probing the speckle to estimate the effective speed of sound, a first step towards quantitative ultrasound imaging.}
\author[JG]{Josselin Garnier \textsuperscript{1}}
\address{\textsuperscript{1}CMAP, CNRS, Ecole polytechnique, Institut Polytechnique de Paris,
91120 Palaiseau, France.}
\author[LG]{Laure Giovangigli \textsuperscript{2}}
\author[QG]{Quentin Goepfert \textsuperscript{2}}
\address{ \textsuperscript{2}POEMS, CNRS, Inria, ENSTA Paris, Institut Polytechnique de Paris,
91120 Palaiseau, France}
\author[PM]{Pierre Millien \textsuperscript{3}}
\address{ \textsuperscript{3}Institut Langevin, ESPCI Paris, PSL University, CNRS, 1 rue Jussieu,  75005 Paris, France.}
\email[Corresponding Author]{pierre.millien@espci.fr}
\begin{document}
\pagestyle{plain}
\maketitle

\begin{abstract}
In this paper, we present a mathematical model and analysis for a new experimental method [Bureau and al., arXiv:2409.13901, 2024] for effective sound velocity estimation in medical ultrasound imaging.  We perform a detailed analysis of the point spread function of a medical ultrasound imaging system when there is a mismatch between the effective sound speed in the medium and the one used in the backpropagation imaging functional.  Based on this analysis, an estimator for the speed of sound error is introduced. Using recent results on stochastic homogenization of the Helmholtz equation, we provide a representation formula for the field scattered by a random multi-scale medium (whose acoustic behavior is similar to a biological tissue) in the time-harmonic regime.  We then prove that  statistical moments of the imaging function can be accessed from data collected with only one realization of the medium.  We show that it is possible to locally extract the point spread function from an image constituted only of speckle and build an estimator for the effective sound velocity in the micro-structured medium.  Some numerical illustrations are presented at the end of the paper. 

\end{abstract}

\begin{center}

\def\keywords2{\vspace{.5em}{\textbf{  Mathematics Subject Classification
(MSC2020).}~\,\relax}}
\def\endkeywords2{\par}
\keywords2{35R30, 92C55,35R60,35J05 }\\

\center
\def\keywords{\vspace{.5em}{\textbf{ Keywords.}~\,\relax}}
\def\endkeywords{\par}
\keywords{Helmholtz equation,  quantitative ultrasound imaging,  stochastic homogenization.}
\end{center}

\section{Introduction}
\paragraph{Scientific context}
Ultrasound imaging is a cheap,  safe portable way of imaging soft tissues in real time. 
Medical ultrasound exams account for about $20\%$ of the imaging exams \cite{england2020diagnostic},  and can be used for numerous different applications,  from obstetric to cardiology or dermatology \cite{ranganayakulu2016ultrasound}.  The principle is the following:  an ultrasonic wave is emitted by a probe to insonify a specific organ (womb, liver, brain...).  The inhomogeneities in the tissues produce echoes (\emph{backscattered wave}) that are measured on the probe.  An image showing the localization of the sources of the echoes is computed using a \emph{backpropagation} type algorithm (Sum and delay,  Kirchoff migration\ldots). These types of \emph{time reversal} based algorithms are all relying on the principle of a \emph{space-time} correspondence (travel time) and require some \emph{a priori} hypothesis on the speed of sound in the medium.  The speed of sound contrast between different types of soft tissues rarely exceeds $10\%$ and in practice most commercial devices use a constant value for the speed of sound, usually close to the one observed in water.  
  Nevertheless this approximation can deteriorate the quality of the images significantly in certain situations by introducing artefacts,  distorting geometric features or reducing the contrast.  
Beyond the improvement of the images, the knowledge of the speed of sound in certain organs can be used as an indicator of a pathological state (such as hepatic steatosis \cite{burgio2019ultrasonic}).  We refer to the review \cite{cloutier2021quantitative} for a complete overview of the current state of the art on quantitative ultrasound imaging technics.   
The stakes of measuring \emph{in situ} the propagation speed of ultrasounds in soft tissues is therefore twofold: improve the reliability of the ultrasound images and provide new sensitive diagnostic tools. 
\paragraph{The reflexion matrix approach}
An important evolution in medical ultrasound imaging was the development of new insonification processes. 
By using a spherical wave or a plane wave,  it is possible to illuminate the whole medium with one wave and to get confocal images in post-processing : instead of focusing physically in the medium like in conventional ultrasound imaging modalities,  the focusing is done numerically by exploiting the linearity of the wave equation and the superposition principle.This method has allowed for a drastic improvement in the frame rate (from $24$fps up to $1000$ fps) with no loss in contrast or resolution \cite{tanter2014ultrafast}.  
Another advantage is that since the image processing is done numerically,  any extra information about the speed of sound can be incorporated \emph{a posteriori} to improve the images.  Speed of sound estimation is an active area of research and several groups have proposed innovative methods based on adjusting the speed of sound by maximizing a \emph{quality indicator},  such as echo amplitude \cite{yamaguchi2021basic},  a coherence factor \cite{imbault2017robust,imbault2018ultrasonic},  or the brightness of a strong reflector \cite{pirmoazen2020quantitative}.  The \emph{computed ultrasound tomography in echo mode} (CUTE) method \cite{jaeger2015computed,jaeger2022pulse,beuret2023windowed} uses a localized detection of aberration phase shifts under different angles of illumination and inverts a linear model that relates these measurements to the speed of sound in tissues. 

Over the last few years, the team of A.  Aubry and M. Fink have obtained spectacular experimental results in the domain of aberration correction and speed of sound estimation using a \emph{reflection matrix} approach \cite{aubry2020reflection,lambert2022,lambert2022ultrasound2} (see section \ref{sec:measurements}).  In \cite{bureau2024reflection} a new method to estimate the speed of sound in soft tissues has been introduced.  The idea,  inspired from adaptative optics,  is also to assess  the speed of sound in the medium by maximizing a quality indicator (the amplitude at the center of the focal spot).  The main strength of this method is that it does not require the presence of a strong isolated reflector (guide star).  The idea is to probe the point spread function directly in the speckle part of the image by using a clever focusing sequence that mimics the presence of an isolated reflector. 
The goal of the present work is to present a robust mathematical model and analysis to support their experimental proof of concept of the method and quantify precisely some of the effects observed experimentally. 

\paragraph{Outline of the article and main contributions}
In this paper,  we start by presenting a general model for medical ultrasound imaging set ups.   
Particular attention is given to incorporating the specific features of these experiments,  whether it is in the description of the acoustic properties of soft tissues with a stochastic model or the use of an asymptotic regime that describes accurately the different scales inherent to the physics of the problem. 
Building on our recent results on the high order stochastic homogeneization of Helmholtz equation \cite{garnier2023scattered} we give an approximation of the measurements in the time-harmonic regime for a random multi-scale medium. 
In section \ref{sec:imagingfunc}, we perform an asymptotic analysis of the point spread function of a medical ultrasound imaging system in the case where the speed of sound used in the backpropagation imaging functional is not the effective speed of sound in the medium.  We define a general notion of \emph{focal spot} and quantify the effects of the misfit in speed of sound on the focal spot (displacement,  deformation\ldots). Based on this analysis, an estimator for the speed of sound error is introduced (proposition \ref{prop:amplitude}).
In section \ref{sec:speedofsoundspeckle},  we perform a mathematical analysis of the method developped by A. Aubry and his collaborators \cite{bureau2024reflection}.  We show that in a micro-structured medium,  it is possible to access the average of the point spread function over the focal spot via the variance of the value of the imaging function at well chosen pixels (lemma \ref{lem:incoh}).  We exhibit an estimator for the effective speed of sound (proposition \ref{prop:esgtimatorvariance}).  In section \ref{sec:ensembletospatial}, we show how to approximate statistical moments of the imaging function with data collected from a single realization of the medium using stationarity properties of the medium.  Proposition \ref{prop:quantificationvariance} gives a quantification of the approximation error for the variance of the imaging function.  The last section includes some numerical simulations of the problem to illustrate the claims of the paper.  

\paragraph{Position with respect to the state of the art}
In the mathematics community,  the problem of velocity estimation from boundary measurements for various types of waves has been an active area of research for decades.  The usefulness of any method depends strongly on the practical experimental constraints.  For example, methods based on the inversion of the Eikonal equation \cite{stefanov2019travel} are of little use in situations where the sensors cannot surround the medium to image.  Similarly,  powerful \emph{full wave inversion} methods \cite{faucher2020adjoint} that can perform remarkably to reconstruct the velocity of seismic waves inside earth's crust will fall short if the imaging process is required to be in real time, due to their algorithmic complexity.  The rule of thumb is that the more \emph{a priori} information a method can incorporate,  the better it will perform in specific situations.  For example,  in the case of medical ultrasound imaging,  the data that are measured (back-scattered field) are generated by the biological tissue's micro-structure, while the properties of interest (effective speed of sound and travel times) appear at a macroscopic scale.  
To the best of our knowledge,  there is no recent mathematical theory studying the effective velocity estimation problem from  boundary measurements in a reflection geometry for a random micro-structured medium, particularly using reconstruction methods with a low algorithmic complexity. 

\section{Model for medical ultrasound experiments}

\subsection{Geometry and governing equations}\label{sec:geometry}
The general model for the governing equation is a  divergence form wave equation in an unbounded domain. 
\paragraph{The coefficients}

We consider a $d$-dimensional framework. Take a simply connected open bounded domain $D$ with at least Lipschitz boundary regularity,  modeling a medium embedded in an infinite homogeneous isotropic medium. 
 Introduce the two coefficients:
\begin{equation}
\left\{\begin{aligned}
a(x) := & a_m  \mathbbm{1}_{\R^d \setminus \overline{D}}(x) + a_D(x)\mathbbm{1}_{D}(x),\\
 n(x) := & n_m \mathbbm{1}_{\R^d \setminus \overline{D}}(x) + n_D(x)\mathbbm{1}_{D}(x),
\end{aligned}\right.  \qquad x \in \R^d,
\end{equation}
where $a_D,n_D\in L^\infty(D)$.  Moreover, we assume that the functions $a$ and $n$ are  uniformly positive.  
In acoustics (which is the context of medical ultrasound imaging),  the coefficients $a$ and $n$ are related to the inverse of the density (usually denoted $\rho$) and the inverse of the bulk modulus (usually denoted $\kappa$).  See for instance \cite[section 3.3.3]{pierce2007basic}. 
For simplicity,  in this paper we make the \emph{restrictive assumption} that \begin{align*}
a_D\equiv a_m =1.
\end{align*}
\begin{remark}This assumption seems restrictive at first as the model would only be accurate for a medium with a constant density.  Moreover,  mathematically,  the representation formulas for the scattered field are very different when the contrast is in the divergence part of the equation.  However,  this is a common hypothesis in the medical ultrasound community, where most models do not have contrast in the divergence.  
This is motivated by the fact that when the scatterers are dilute (at distances larger than the wavelength) and do not exhibit a particular orientation or elongation,  it is impossible to distinguish a dipole from a monopole in the reflection geometry.   Note that outside the dilute case the question \emph{under which conditions are  the two contrast models equivalent ?} remains open. 
\end{remark}
\paragraph{Time domain versus frequency domain}
All the experiments are done in the time domain,  and backscattered echoes are measured as time-dependent functions.  Nevertheless, besides the simple \emph{sum and delay} algorithms \cite{perrot2021so},  most of the advanced post-processing is usually done in the frequency domain. The introduction of the notion of wavelength is critical to quantify the resolution of the methods so our mathematical analysis will be done in the time harmonic domain.  A Laplace-Fourier transform can be used to switch from one domain to another.  In the rest of the chapter the frequency will be denoted by $\omega$.  Since the transducers have a finite bandwith $\mathcal{BW}:=[\omega_-,\omega_+]$ we will always assume $\Re \omega \in \mathcal{BW}$.
We will only write the dependency of the wavefields with respect to $\omega$ when it is needed.  Usually,  the time profile of the incident wave at a given point has a central frequency $\omega_c$:
\begin{align*}
\mathcal{BW} = [\omega_c-\mathcal{B},\omega_c+\mathcal{B}],
\end{align*} where $\mathcal{B}$ is the bandwidth.  If the pulse consists of a few oscillations at the central frequency (see an example fig. \ref{fig:InputSignal}) then we are in a \emph{broadband} situation $\mathcal{B}=\mathcal{O}(\omega_c)$. If the pulse consists of oscillations at the central frequency with an enveloppe decaying \emph{slowly} compared to the central period then we are in a \emph{narrowband} case $\mathcal{B}\ll \omega_c$.

\paragraph{The incident wave}
The choice of the incident wave $u^i$ has little impact on the mathemati-cal model for the direct problem but is of crucial importance in the applications and for the inverse problem. 
The simplest model for an incident wave is to consider an eigenfunction for the free space laplacian, \emph{i.e. } a plane wave of amplitude $U^i$ propagating in direction $\theta\in \mathbb{S}^{d-1}$ with wavenumber $k_m=\sqrt{ n_m} \omega$:
\begin{align*}
u^i(x)=U^i e^{ik_m \theta \cdot x},\qquad x\in \R^d.
\end{align*}
This is a good model when the source is far enough from the domain $D$ to be considered \emph{at infinity}. 

It is possible to use more sophisticated models for the incident wave.  For example a point source outside of the domain $\overline{D}$ or a finite sum of point sources.  In this case the incident wave is expressed as a linear combination of Green functions: 
\begin{align*}
u^i(x)= \sum_{j=1}^N \alpha_j \Gamma^{k_m}(x,x_j),\qquad x\in \R^d\setminus\cup_j\{x_j\},
\end{align*}
with $\alpha_j\in \mathbb{C}$. 
This is a good model for the medical ultrasound experiments, where probes, which can be close to the medium, are made of a finite number of transducers.
When the incident wave is emitted via a point source located at emission point $x_e$ we will denote it $u^i(x_e,\cdot)$ to keep track of the dependency on the location of the source when it is needed,  especially for the inverse problem.  
Sometimes though a continuous model for the probe can be more suited for a mathematical analysis.  In this case,  one can model the probe by a compact manifold $\mathcal{P}$ outside of $\overline{D}$ and the incident wave can be modeled by a single or a double layer potential on that manifold:
\begin{align*}
u^i(x) = \mathcal{S}_\mathcal{P}^{k_m}(\varphi)(x) \qquad \mathrm{or} \qquad u^i(x) = \mathcal{D}_\mathcal{P}^{k_m}(\psi)(x) \qquad  x\in \R^d\setminus\mathcal{P},
\end{align*}
where the layer potentials are given by \cite{ammari2009layer}:
\begin{align*}
\mathcal{S}_\mathcal{P}^{k_m}(\varphi)(x) = \int_\mathcal{P} \varphi(y) \Gamma^{k_m}(x,y) \dd \sigma(y)  \qquad \mathrm{and} \qquad \mathcal{D}_\mathcal{P}^{k_m}(\psi)(x)=\int_\mathcal{P} \varphi(y)\frac{ \partial\Gamma^{k_m}}{\partial \nu(x)}(x,y) \dd \sigma(y).
\end{align*}
The profile of the incident wave in the time domain $f(t)$ can be easily added to this model by adding its Fourier-Laplace transform $\hat{f}(\omega)$ as a prefactor.  The frequency content does not affect much the direct problem but has a major impact on the reconstructed images' resolution as will be detailed in section \ref{subsec:propertiespsf}

\paragraph{The governing equation}
The scattering problem by the medium $D$ is modeled by the following Helmholtz equation in $\R^d$:
\begin{align} \label{eq:divformhelmholtz}
\left\{\begin{aligned}
&\Delta u  + \omega^2 n u =0 \qquad \mathrm{in\,}  \R^d\\
&\lim_{\vert x\vert \rightarrow\infty} \left\vert x\right\vert^{\frac{d-1}{2}} \left(\partial_{\vert x\vert} \left(u-u^i\right) - ik_m \left(u-u^i\right) \right) =0. 
\end{aligned}\right. 
\end{align}
This equation has a unique solution in $H^2_{\mathrm{loc}}(\R^d)$. The well-posedness is well known and can be found in textbooks \cite{PJoly}.
The difficulty of equation \eqref{eq:divformhelmholtz} lies in getting either sharp regularity results for the solution or some explicit control with respect to the frequency.  Usually some extra assumptions on the coefficients are needed.  
 These questions have been a very active area of research over the last decades. We refer to the excellent paper \cite{moiola2019acoustic} and references within for a recent overview of these questions.

\subsection{Measurement model : the reflection matrix}\label{sec:measurements}
As mentionned previously,  in practice the probe has a finite number $N$ of transducers that can be modeled as a finite collection of point sources $(x_i)_{i=1\ldots N}$ each on the manifold $\mathcal{P}$ describing the shape of the probe. 
For medical imaging modelization it is reasonable to assume the probe is linear and positioned just above the medium.  We set
\begin{align*}
\mathcal{P}=[-\ell,\ell]^{d-1}\times\left\{0\right\}.
\end{align*} 
 Each transducer can be controlled independently.  The measurements are obtained in the following way:
\begin{itemize}
\item one transducer at $x_{e}$ is activated  to emit an incident wave $u^i(x_e,\cdot,\omega)$,
\item the scattered wave $u^s(x_e,\cdot, \omega)=u(x_e,\cdot, \omega)-u^i(x_e,\cdot, \omega)$ is recorded by all the transducers  $x_r$ on the probe for all frequencies in the bandwidth,
\item the operation is repeated with another transducer $x_{e'}$ and so on.
\end{itemize}
In the end one has access to the measurement matrix at each frequency $$M_{e,r}(\omega)=u^s(x_e,x_r,\omega) \in \C^{N\times N}.$$
The continuous model to describe this type of measurements is to assume we are given the measurement map:
\begin{align*}
M\in L^2(\mathcal{P}\times\mathcal{P}\times \mathcal{BW})
\end{align*}
given by
\begin{align}\label{eq:defmeasurementmaps}
M(x_e,x_r,\omega) = u^s(x_e,x_r, \omega).
\end{align}
\begin{remark}
In order to improve the signal-to-noise ratio in the presence of additive noise,
multiplexing measurement strategies can be used \cite{harwit2012hadamard}, in which the transducers transmit orthogonal waveforms and the measurement matrix can then be extracted numerically.
\end{remark}

\subsection{Model for the biological tissues}\label{subsec:modelbio}

We now want to give a more precise model for the coefficient $n$ inside the medium $D$. 
The signal measured by the probe comes from the local discontinuities in the acoustic impedance.  To correctly capture the nature of the experiment,  the model for the coefficient has to take into account both types of local discontinuities:  the ones at the interface between different types of soft  tissues (fat, liver, muscle etc\ldots) and the ones occurring inside each tissue type due to the presence of small unresolved scatterers.  
The model was introduced in \cite{garnier2023scattered}, we present a simple version that we will use in this work.

\paragraph{Simple first model}
In this context, the propagating medium could be described as a  composite medium,  with a few components, each describing a different tissue type.  Then the presence of small unresolved scatterers could be modeled by \emph{randomly sprinkling} each component with a large but finite number of small or point-like scatterers. 
This looks like a perfectly sound model and has been successfully used to describe blood flow imaging experiments \cite{UltrafastImaging}. If there are enough small scatterers the numerical simulation of the propagation in such a medium will look like an ultrasound imaging experiment.  Most estimates and tools available for the study of composite mediums are valid for a fixed number of scatterers and we do not know how the constants in the different estimates will behave when the number $N$ of scatterers grows very large. 
Moreover,  the state of the art models do not account for the fact that a change in the properties of the scatterers (their density in the medium, their contrast\ldots) induces a change in the acoustic properties of the medium at the macroscopic scale. 
A good way to convince oneself of this is to look at how the experimental phantoms are made.  They usually consist of an incompressible gel with acoustic properties similar to those of water, with microscopic acoustic grains embedded inside the gel everywhere.  Variations at the macro scale in acoustic properties are obtained by changing locally the repartition of the unresolved acoustic grains inside.  The other strong hint towards the fact that macro-scale acoustic properties of tissues can be derived by an homogeneization process is the experimental observation of anisotropy.  Since the compressibility and density are both scalar parameters,  an homogeneization process seems like a necessary way to understand the acoustic anisotropy that can sometimes be observed in certain biological tissues like muscles. 

\paragraph{Advanced model}
In light of the previous paragraph, we aim at providing a model that can link the properties of the scatterers to the effective acoustic parameters of the medium.  The good way to do that is to use a model that preserves the proportion of volume occupied by the scatterers even when the size of the scatterers decreases.  

We start by considering a random distribution of scatterers of characteristic size $r_0>0$ in $\R^d$  and of average density of order one.  We then rescale this distribution using a scaling parameter $\varepsilon>0$ to place the scatterers inside the medium $D$. 

Let $(x_i)_{i \in \mathbb{N}}$ be the point process in $\R^d$ corresponding to the centers of the scatterers.  Consider an open connected domain $\mathcal{Q}$ of measure $r_0^d$ centered at $0$.  Each scatterer is then represented by $S_i=\left\{ x_i + \mathcal{Q}\right\}$. 
We make the following assumptions on $(x_i)_{i \in \mathbb{N}}$:
\begin{itemize}
\item[-]$(x_i)_{i \in \mathbb{N}}$ is stationary, \textit{i.e.} its probability distribution is invariant by translation and ergodic;
\item[-] the scatterers lie at a distance at least $\delta_{min} > 0$ from one another, \textit{i.e.} there exists $\delta_{min} > 0$ such that  $$\forall i \neq j,~dist(s_i, s_j) > \delta_{min}\quad \textrm{a.s.} $$ 
\end{itemize} 
We denote by $\dst S= \cup_i \left\{ x_i + \mathcal{Q}\right\}$.
 We introduce the parameter $n: \R^d \times\Omega \to [n_-,n_+]$
\begin{equation} \label{eq:AN}
 n:= n_0\mathbbm{1}_{\R^d \setminus \overline{S}} + \sum_i n_{S_i}\mathbbm{1}_{S_i},
\end{equation}
where $(n_{S_i})_i$ are independent and identically distributed.  
We then fix $\varepsilon>0$ such that $\varepsilon r_0$ is now a length representing the characteristic size of the scatterers in the medium.  
The natural definition for the parameter $n_D$ in the domain $D$  in the case where scatterers are of size of order $\varepsilon$ is:
\begin{align*}
 n_D(x) =n_\varepsilon(x):=  n\left(\frac{x}{\eps}\right). 
\end{align*}

\begin{remark} Even though the characteristic size of the scatterers is going to zero as the scaling parameter $\varepsilon$ is going to zero,  the volume fraction they occupy in $D$ stays constant as the number of scatterers inside $D$ grows as $\vert D\vert r_0^{-d}\varepsilon^{-d}$.  This is the main difference with the other \emph{small scatterer} models like the Foldy-Lax model.  
\end{remark}

\begin{remark}
There is no need to take the same shape for all the scatterers.  The shape $\mathcal{Q}$ can be replaced by $\mathcal{Q}_i$ where the shapes $\mathcal{Q}_i$ are realizations of an \emph{i.i.d}. process.
\end{remark}

\begin{remark} Our results apply to any  random medium described by an index of the form $f_\varepsilon(x)= f(\frac{x}{\varepsilon})$  where $f:\R^d \times \Omega \to [n_-,n_+]$ is stationary and verifies \eqref{eq:Decaying}. 
\end{remark}

We assume moreover that the distribution of scatterers possess some decorrelation properties.  Namely,  let us denote by $C$ the covariance function of $n$:
\begin{equation}
C(x) := \Cov(n(x), n(0)),  \qquad \forall x\in \R^d.
\end{equation}
We assume that $C$ can be controlled by a bounded sufficiently decaying decreasing function $\Phi:\R^+\to\R^+$ \emph{i.e} such that for some $p > 2(d-1)$ (see remark \ref{rem:remarkA1}):
\begin{equation} \label{eq:Decaying}
\int_{\R^d} \Phi(\vert x\vert )|x|^p \dd x< \infty.
\end{equation}
In particular,  $\Phi$ and $\Phi^{\frac{1}{2}}$ are in $L^1(\R^+)$. 

\begin{figure}

\center
\includegraphics[scale=2]{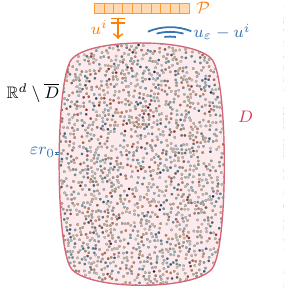}\caption{Illustration of the geometry for an ultrasound imaging experiment. }

\end{figure}

\subsection{Homogenized representation formula}
The corner stone of our analysis is the simplified asymptotic expression for the measurements obtained with quantitative stochastic homogeneization techniques in our previous paper \cite{garnier2023scattered}.  More precisely,  we use the integral representation formula of \cite[theorem 10]{garnier2023scattered}. 
We introduce $G^\star$ the Green function associated to the homogenized problem \textit{i.e.} $G^\star$ is the outgoing solution  in $\mathcal{D}'\left(\R^d\right)$ of all 
\begin{equation}\label{eq:Gstar}
 -\Delta G^\star(\cdot, y) - \tilde{n}^\star  \omega^2G^\star(\cdot, y) = \delta(\cdot - y)  \qquad \textrm{in}\, \R^d,
\end{equation}
with  $\tilde{n}^\star$ being the homogeneized coefficient in the object $D$ and the coefficient of the surrounding medium outside of $D$:  \begin{equation}
 \tilde{n}^\star(x) :=  n_m \mathbbm{1}_{\R^d \setminus \overline{D}}(x) + n^\star\mathbbm{1}_{D}(x), \qquad x \in \R^d.
\end{equation}

For the simplicity of the computations we can make the following assumption that $n_0 = n_m$ and that $\mathbb{E}[n_S]=0$ so that $n^\star=n_0$, $\tilde{n}^\star \equiv n_0$,  and $G^\star= \Gamma^{\frac{\omega}{c^\star}}$ with $c^\star = \frac{1}{\sqrt{n^\star}}$. In that case the coherent wave is the incident field $\Gamma^{\frac{\omega}{c^\star}}(\cdot,x_e)$.  This can be viewed as an effective matched boundary condition.  
\begin{remark}This condition is restrictive in theory but mildly restrictive in practice.  In case of a mismatch of the effective speed of sound,  some reflections will occur at the boundary of the domain.  The first one from the top of the domain $D$ can be filtered in the time domain. The ones occurring at the side of the domain will not pollute too much the measurements in the geometry considered, due to  the directivity of the sensors.  
\end{remark}

We recall that in that case, the scattered field can be approximated by \cite[theorem 10]{garnier2023scattered} 
\begin{align}\label{eq:homogenizedLS}
 \mathcal{U}^s(x_e,y,\omega) := \omega^2 \int_D (n_\varepsilon(x)-n^{\star} ) \Gamma^{k^\star}(x,x_e) \Gamma^{k^\star}(y,x)\dd x, \qquad y\in \R^d \setminus\overline{D},
\end{align}
and the pointwise error control is given in proposition \ref{prop:homolargedomain}.

\subsection{Asymptotic regime}\label{sec:asymptotic}
\paragraph{Typical values for the parameters}

In practical situations, the typical central frequency of the transducers is of the order of a few Mhz, the speed of sound varies from $1400 - 1600$ $m\text{s}^{-1}$, and thus the wavelength is of order $\lambda \sim 1mm $   while the typical diameter of $D$ is of order $10cm$.  Finally the size of the transducer array is usually a few centimeters $2\ell \sim 4cm$(see \cite[Table~2.2]{phdFlavien}). 
The scatterers are usually of typical size $\varepsilon r_0$ of a few micron (think for example of a red blood cell),  around two order of magnitude below the central wavelength.  
Mathematically this can be expressed by
\begin{align*}
\varepsilon r_0 \ll \frac{c^\star}{\omega_c} \ll \ell  \ll  \text{diam}(D).
\end{align*}
\paragraph{The paraxial regime}
In order to get analytic expressions,  we will consider a classical asymptotic regime called the \emph{paraxial regime}  \cite[chapter 6]{garnier2016passive}. This will be done by introducing a small scaling parameter denoted $\eta\ll 1$ and expressing the geometric quantities of the model as functions of $\eta$.
We introduce
\begin{align*}
\omega_c=:\frac{\omega_0}{\eta}, \quad \mathcal{B}=:\frac{\mathcal{B}_0}{\eta}, \quad \mathcal{P} =: \eta^{\frac{1}{2}}\mathcal{P}_0, \quad \ell=:\eta^{\frac{1}{2}}\ell_0,  \quad \varepsilon =: \eta^{\alpha}, \quad \alpha>0. 
\end{align*}
Regarding the points of the domain,  we will say $y\in D$ satisfies the paraxial approximation if its transverse and axial coordinates can be written as:
\begin{align*}
y := \left(\eta^{\frac{1}{2}}y^\perp,y^\shortparallel\right)\qquad \, y^\perp \in \R^{d-1}, y^\shortparallel\in \R^+. 
\end{align*}

For the reminder of the manuscript the notation $y^\perp$ will designate the projection on the transverse component rescaled with respect to the asymptotic parameter $\eta$: 
\begin{align*}
\R^d \longrightarrow & \R^{d-1}\\
y \longmapsto & y^\perp := \eta^{-\frac{1}{2}}\mathbf{P}^\perp (y) 
\end{align*}
where $\mathbf{P}^\perp$ is the projection on the transverse plane.

\begin{remark} \label{rem:alpha}For the approximation of eq. \eqref{eq:homogenizedLS} to be accurate the error term of equation \eqref{eq:homogN} in appendix \ref{appendix:random} needs to be small,  \emph{i.e.} goes to $0$ when $\eta$ goes to zero.   This imposes a lower bound on $\alpha$.  The bound depends on the dimension $d$ and the exponent $p'$ introduced in the appendix.  In dimension $d=3$,  $\alpha> \frac{5}{2}$ is a sufficient condition. 
\end{remark}

\section{The imaging functional}\label{sec:imagingfunc}
\subsection{The classical confocal imaging function} \label{sec:confocal}
In the case where the effective speed of sound $c^\star$ in the medium $D$ is known the classical confocal imaging function can be written:
\begin{align}
I^{c^\star}(z) := \int_{(\mathcal{P} \times \mathcal{P} \times \mathcal{B})} \overline{M(x_e,x_r, \omega)} G^\star(z,x_e) G^\star(z,x_r) \mathrm{d}x_e \dd x_r \dd\omega.\label{eq:kernelI} \qquad z\in D
\end{align}

This function,  also known as the \emph{Kirchhoff migration} function has been extensively studied in various cases (see \cite{fouque2007wave} and references within).

In this context it is straightforward from the representation formula \eqref{eq:homogenizedLS} that the measurement matrix $M$ can be written as
\begin{align}\label{eq:integralM}
M(x_e,x_r,\omega) = \omega^2 \int_D  \left(n(y)-n^\star\right) G^\star(x_r,y)G^\star(x_e,y)\dd y. 
\end{align}
 \begin{remark} Even though $I^{c^\star}$ is defined for $z\in D$ it is important to keep in mind that the imaging domain is a \emph{virtual domain} while the integral in equation \eqref{eq:integralM} happens physically in the real object $D$.   This will be of importance in the rest of the work and justifies the following definition to avoid the confusion between the imaging points and the physical points in the object. 
\end{remark}
\begin{definition}[Virtual Domain] The definition domain for $I^{c^\star}$ will be denoted $D'$. 
\end{definition}
The imaging function can then be expressed as
\begin{align*}
I^{c^\star}(z) = \int_D  \left(n_\varepsilon(y)-n^\star\right) F^{c^\star} (z,y) \dd y,  \qquad z\in D'
\end{align*}
which is an integral operator acting on $n_\varepsilon-n^\star$ with kernel: \begin{align*}
 F^{c^\star} (z,y)=\int_{\mathcal{B}} \omega^2\left(\int_{\mathcal{P}} \Gamma^{\frac{\omega}{c^\star}}(z,x_r)\overline{G^{\star}(y,x_r)} \dd x_r\right)^2\dd \omega, \qquad (z,y)\in D'\times D.
 \end{align*}
 which acts as the so-called \emph{point spread function} for the imaging system.

\subsection{Imaging with an unknown speed of sound}
In our case the effective speed of sound in the medium $D$ is unknown and we can parametrize the imaging function by the speed of sound  $c$ used in the backpropagation.  
\begin{definition}
\begin{align}
I^c(z):= \int_{(\mathcal{P} \times \mathcal{P} \times \mathcal{B})} \overline{M(x_e,x_r, \omega)} \Gamma^{\frac{\omega}{c}}(z,x_e) \Gamma^{\frac{\omega}{c}}(z,x_r) \mathrm{d}x_e \dd x_r \dd\omega,\qquad z\in D'.\label{eq:defIc}
\end{align}
\end{definition}
\begin{remark}
The measurements are done in the time domain on a finite interval $[0,T]$. The imaging domain $D'$'s  axial size is determined by the longuest distance a wave can travel in $\frac{T}{2}$.
Note that now that the speed of sound in the virtual domain is modified the domain $D'$ is now different from $D$. The explicit expression for the mapping $\varphi_c$ between $D$ and $D'$ is given in lemma \ref{lem:PSFparaxial} (eq.  \eqref{eq:defphi}). \end{remark}
We can still express the imaging functional as an integral operator
\begin{align*}
I^c(z) = \int_D  \left(n_\varepsilon(y)-n^\star\right) F^{c} (z,y) \dd y
\end{align*}
acting on $n_\varepsilon-n^\star$ with kernel: \begin{align*}
 F^{c} (z,y)=\int_{\mathcal{B}}\omega^2 \left(\int_{\mathcal{P}} \Gamma^{\frac{\omega}{c}}(z,x_r)\overline{G^{\star}(y,x_r)} \dd x_r\right)^2\dd \omega \qquad z,y\in D'\times D.
 \end{align*}
 which is the new \emph{point spread function} for the imaging system. 
 
 \subsection{Properties of the point spread function}\label{subsec:propertiespsf}

\subsubsection{Focal spot in the case $c=c^\star$}
The point spread function at the correct backpropagation speed $F^{c^\star}(z,y)$ has already been extensively studied in that setting (see for instance \cite{fouque2007wave}).  For a fixed $y_0\in D$, the function 
\begin{align*}
D'&\longrightarrow \mathbb{C}\\z&\longmapsto F^{c^\star}(z,y_0)
\end{align*} presents a peak centered at $y_0$ : its amplitude is of order $\frac{1}{\eta}$ in a small region around $y_0$ called the \emph{focal spot} and negligible outside of this region.  The size of the focal spot depends on the geometry of the problem : the central wavelength, the bandwidth, the numerical aperture...
The point spread function is of course symmetrical in that case since $D=D'$ and $F^{c^\star}(z,y)=F^{c^\star}(y,z)$. 
See fig. \ref{fig:psf}, middle insert, for an exemple corresponding to the geometry of medical ultrasound imaging methods. 

We can introduce a more precise definition of the focal spot.  For this we need to introduce the family of level sets in the image domain:
\begin{align*}
L_\Theta (y_0)\ := \left\{ z\in D',  \left\vert F^{c^\star}(z,y_0)\right\vert =\Theta\right\}.
\end{align*}
Now we consider the family of nested and connected sets $\mathcal{D}'(y_0,\Theta)$ defined as the \emph{minimal volume} with $\mathcal{C}^{0,\alpha}$ boundary (with $\alpha >0$) containing the level set $L_\Theta (y_0)$, \emph{i.e.}
\begin{align*}
\mathcal{D} '(y_0,\Theta):= \underset{  L_\Theta  \subset \Omega'\subset D', \, \Omega' \ni y_0}{\mathrm{argmin\, }} \left\vert \Omega'\right \vert.
\end{align*}
The focal spot can now be defined:
\begin{definition}[Focal spot at threshold $\delta$]
Consider  $0<\delta\ll 1$ and $y_0 \in D$. There exists a $\Theta(\delta)>0$ such that \begin{align*}
 \int_{D'\setminus \overline{\mathcal{D}'(y_0,\Theta(\delta))}}\left\vert F^{c^\star}(z,y_0)\right\vert \dd z <\delta \int_{D'}\left\vert   F^{c^\star}(z,y_0)\right\vert \dd z.
\end{align*}
We denote $\mathcal{D}'_\delta(y_0):= \mathcal{D}'\left	(y_0,\Theta(\delta)\right)$.
\end{definition}

\begin{remark} The existence of $\mathcal{D}'_\delta(y_0)$ is obvious as $F\in L^1(D'\times D)$.  Uniqueness comes from the uniqueness of $\mathcal{D}'(y_0,\Theta)$ for a given $\Theta$ and the monotonicity of the function $$\dst \Theta \mapsto \int_{D'\setminus \overline{\mathcal{D}'(y_0,\Theta)}}\left\vert F^{c^\star}(z,y_0)\right\vert \dd z.$$Such a definition only has a practical interest because $F(\cdot,y_0)$ is a \emph{peaked} function  and therefore in practice the domain $\mathcal{D}'_\delta$ is small compared to the size of the domain, see remark \ref{rem:sizefocal}. 
Therefore $\dst \int_{\mathcal{D}'_\delta(y_0)}\left\vert F^{c^\star}(z,z^\star)\right\vert \dd z$ is a good approximation of $\dst \int_{D'}\left\vert   F^{c^\star}(z,y_0) \right\vert \dd z$.  
\end{remark}
\begin{remark}At this point the reader might wonder what is the necessity for a new technical definition for the focal spot which is  a well known notion in the imaging community.  This definition is actually just an extension of the definition used in the physics community where in practice,  the focal spot is often characterized as the area inside the $-6$dB level set around the peak of the point-spread function. The constraint on the $L^1$ norm is added here for technical reasons,  to be able to approximate quantities of the type $\dst \int_{D'} F^{c^\star}(z,y) f(z) \dd z$ by $\dst \int_{D'_{\delta}(y)} F^{c^\star}(z,y) f(z) \dd z$ for some various $f$. 
\end{remark}

\begin{remark}\label{rem:sizefocal}In practice,  for $\delta \approx 0.1$ the focal spot's sizes are approximatively the Rayleigh's resolution limit and behave like $\frac{c^\star}{\omega_c}$.  More precisely,  in the broadband regime, in the case of a linear array probe the transverse size of the focal spot at $y_0\in D$ is given by:
\begin{align*}
\frac{2 c^\star}{\omega_c} \frac{y_0^\shortparallel}{\ell}
\end{align*}
and its axial size is
\begin{align*}
\frac{2c^\star}{\mathcal{B}},
\end{align*} where $\mathcal{B}$ is the bandwidth and $\frac{y_0^\shortparallel}{\ell}$ is the numerical aperture at depth $y_0^\shortparallel$. See figure \ref{fig:psf}. 
\end{remark}

\subsubsection{Focal spot in the case $c\neq c^\star$}
In the case where $c\neq c^\star$ the function \begin{align*}
D'&\longrightarrow \mathbb{C}\\z&\longmapsto F^{c^\star}(z,y_0)
\end{align*} still has the same qualitative properties as in the previous case : there exists a small region of $D'$ where most of the signal concentrates. 

We can still define a notion of focal spot: 
\begin{definition}[Focal spot at threshold $\delta$ and speed $c$]
Consider  $0<\delta \ll 1$ and $y_0\in D$. There exists a domain $\mathcal{D}^{c}_\delta(\varphi_c(y_0))'\subset D'$  such that \begin{align*}
 \int_{D'\setminus \overline{\mathcal{D}^c_\delta(\varphi_c(y_0))'}}\left\vert F^{c}(z,y_0)\right\vert \dd z <\delta\int_{D'}\left\vert   F^{c}(z,y_0 )\right\vert \dd z ,
\end{align*} where $\varphi_c$ is defined in \eqref{eq:defphi}.
\end{definition}
\begin{remark} A simple interpretation for the focal spot is the following : for a point $y_0\in D$ and a given back-propagation speed of sound $c$, the focal spot $\mathcal{D}'_{\delta}$ indicates the area in the image $D'$ where the contribution of point $y_0$ will appear. 
\end{remark}

The well-posedness of that definition is a direct consequence of the following lemma:
\begin{lemma}[Point spread function in the paraxial approximation]\label{lem:PSFparaxial}
 Let $(z,y)\in D'\times D$ satisfying the paraxial approximation, \emph{i.e} consider $0<\eta\ll 1$ and
\begin{align*}
y=\left( \eta^{\frac{1}{2}} y^\perp,y^\shortparallel \right)\in D, \quad \mathrm{and} \quad  z=\left(\eta^{\frac{1}{2}} z^\perp,z^\shortparallel, \right)\in D'.
\end{align*}
 In the paraxial regime in dimension $3$ the point spread function has the expression:
  \begin{multline}\label{eq:PSF2bis}
 F^{c}(z,y) = \eta^{-1} \left(\frac{c}{c^\star}\right)^2 \frac{\ell_0^4}{\left(16\pi^2|z||\varphi_c(y)| \right)^2}    \int_{\mathcal{B}_0}\omega^2   e^{i\frac{2\omega}{\eta c}\left( |z| - |\varphi_c(y)| \right)} e^{i\frac{\omega}{c}\left(1-\left(\frac{c^\star}{c}\right)^2\right) \frac{\vert \varphi_c(y)^\perp\vert^2}{\vert \varphi_c(y)\vert}}
\\ \mathcal{G}^2\left(\frac{\omega\ell_0}{c}\left( \frac{z^\perp}{|z|}- \frac{\varphi_c(y)^\perp}{|\varphi_c(y)|}\right) , \frac{\omega\ell_0^2}{c}\left(  \frac{1}{|z|} - \left(\frac{c}{c^\star}\right)^2\frac{1}{|\varphi_c(y)|}\right)\right)\dd \omega+ \mathcal{O}(1),
 \end{multline}
 or equivalently
   \begin{multline}\label{eq:PSF2}
 F^{c}(z,y) = \eta^{-1} \left(\frac{c^\star}{c}\right)^2 \frac{\ell_0^4}{\left(16\pi^2|y||\varphi_c^{-1}(z)| \right)^2}    \int_{\mathcal{B}_0}\omega^2   e^{i\frac{2\omega}{\eta c^\star}\left( |\varphi_c^{-1}(z)| - |y| \right)}  e^{i\frac{\omega}{c^\star}\left(\left(\frac{c}{c^\star}\right)^2-1\right) \frac{\vert \varphi_c^{-1}(z)^\perp\vert^2}{\vert \varphi^{-1}_c(z)\vert}}
\\ \mathcal{G}^2\left(\frac{\omega\ell_0}{c^\star}\left( \frac{\varphi_c^{-1}(z)^\perp}{|\varphi_c^{-1}(z)|}- \frac{y^\perp}{|y|}\right) , \frac{\omega\ell_0^2}{c^\star}\left( \left(\frac{c^\star}{c}\right)^2 \frac{1}{|\varphi_c^{-1}(z)|} - \frac{1}{|y|}\right)\right)\dd \omega + \mathcal{O}(1),
 \end{multline}
 where
 \begin{align*}
 \mathcal{G}(\xi_1,\xi_2) := \int_{[-1,1]^2} e^{-i x_e^\perp \cdot \xi_1 + i \frac{\left\vert x_e^\perp\right\vert^2}{2} \xi_2} \dd x_e^\perp \qquad \xi_1 \in \R^2, \xi_2\in \R,
 \end{align*}
and
\begin{align}\label{eq:defphi}
\begin{aligned}
D& \longrightarrow D'\\
y &\longmapsto \varphi_c(y):= \left(\eta^\frac{1}{2} \left( \frac{c}{c^\star}\right)^2y^\perp,  \frac{c}{c^\star} y^\shortparallel\right).
\end{aligned}
\end{align}
\end{lemma}

\begin{remark} In dimension $2$,  \eqref{eq:PSF2bis} and \eqref{eq:PSF2} differ only by the following respective prefactors in the integrand
\begin{align*}
\frac{32\pi \sqrt{\vert z\vert \vert \varphi_c(y)\vert c c^\star}}{\omega},\qquad
\frac{32\pi \sqrt{\vert y\vert \vert \varphi_c^{-1}(z)\vert c c^\star}}{\omega}.
\end{align*}
and $\mathcal{G}$:
 \begin{align*}
 \mathcal{G}(\xi_1,\xi_2) := \int_{[-1,1]^2} e^{-i x_e^\perp  \xi_1 + i \frac{\left\vert x_e^\perp\right\vert^2}{2} \xi_2} \dd x_e^\perp \qquad \xi_1 \in \R, \xi_2\in \R.
 \end{align*}
\end{remark}

\begin{proof}

We suppose that $x_e:=(\eta^{\frac{1}{2}}x_e^\perp, 0)$ and that $z$ is in the paraxial regime $z=(\eta^{\frac{1}{2}}z^\perp,z^\shortparallel)$. We have:
\begin{equation*}
|x_e-z|=|z|\left(1+\frac{\eta}{2}\frac{|x_e^\perp|^2}{|z|^2}-\eta\frac{\langle z^\perp,x_e^\perp\rangle}{|z|^2}\right)+\mathcal{O}(\eta^{2}).
\end{equation*}
For a point $y:=(\eta^{\frac{1}{2}}y^\perp, y^\shortparallel)$:
\begin{multline*}
\frac{c^\star}{c}|x_e-z|-|x_e-y|=\left(\frac{c^\star}{c}|z|-|y|\right)-\eta\left(\frac{c^\star}{c}\frac{\langle z^\perp,x_e^\perp\rangle}{|z|}-\frac{\langle y^\perp,x_e^\perp\rangle}{|y|}\right)\\+\frac{\eta}{2} |x_e^\perp|^2\left(\frac{c^\star}{c}\frac{1}{|z|}-\frac{1}{|y|}\right)+\mathcal{O}(\eta^{2}).
 \end{multline*}
 Since the point-spread function is a peaked function we want to find the center of the focal spot that we denote $\varphi_c(y)\in D'$ and express the point spread function with respect to $z-\varphi_c(y)$. 
 We can find this center $\varphi_c(y)$ by noting that since the frequency behaves as $\eta^{-1}$, a necessary condition is that if $z=\varphi_c(y)$ then the term in $\eta^0$  in the phase must cancel. 
 Cancelling the zero order term gives the first necessary condition for $\varphi_c(y)$  
\begin{equation}\label{eq:varphi1}|\varphi_c(y)|=\frac{c}{c^\star}|y|+\mathcal{O}(\eta).\end{equation}
This gives, since $y$ is in the paraxial regime:
\begin{align*}
\varphi_c(y)^\shortparallel:= \frac{c}{c^\star}y^\shortparallel.
\end{align*}
We are now looking for a necessary condition on the center of the focal spot $\varphi_c(y)$'s transverse coordinate $\varphi_c(y)^\perp$. 
Integrating over the probe and the bandwidth, for $z=\varphi_c(y)$
 \begin{multline*}
 F^{c}(z,y) = \eta^{-1} \left(\frac{c^\star}{c}\right)^2 \frac{\ell_0^4}{\left(16\pi^2|y||\varphi_c(y)| \right)^2}    \int_{\mathcal{B}_0}\omega^2   e^{i\frac{\omega}{c}\left(1-\left(\frac{c^\star}{c}\right)^2\right) \frac{\vert \varphi_c(y)^\perp\vert^2}{\vert \varphi_c(y)\vert}} 
\\ \mathcal{G}^2\left(\frac{\omega\ell_0}{c^\star}\left( \left(\frac{c^\star}{c}\right)^2\frac{\varphi_c(y)^\perp}{|y|}- \frac{y^\perp}{|y|}\right) , \frac{\omega\ell_0^2}{c^\star\vert y\vert }\left( \left(\frac{c^\star}{c}\right)^2 - 1\right)\right)\dd \omega+ \mathcal{O}(1),
 \end{multline*}
 where
 \begin{align*}
 \mathcal{G}(\xi_1,\xi_2) := \int_{[-1,1]^2} e^{-i x_e^\perp \cdot \xi_1 + i \frac{\left\vert x_e\right\vert^2}{2} \xi_2} \dd x_e^\perp \qquad \xi_1 \in \R^2, \xi_2\in \R.
 \end{align*}
Since $\mathcal{G}$ is a peaked function at the origin for both arguments,  this gives us the transverse position of the center of the focal spot by setting:
 \begin{align*}
\left(\frac{c^\star}{c}\right)^2\frac{\varphi_c(y)^\perp}{|y|}- \frac{y^\perp}{|y|}=0
 \end{align*}
 and we get the second necessary condition:
 \begin{equation}\label{eq:varphi2}\langle \varphi_c(y)^\perp,x_e^\perp\rangle=\left(\frac{c}{c^\star}\right)^2\langle y^\perp,x_e^\perp\rangle.
\end{equation}
 We therefore set
\begin{equation*} \varphi_c(y):=\left(\eta^{\frac{1}{2}}\left(\frac{c}{c^\star}\right)^2y^\perp, \frac{c}{c^\star}y^\shortparallel\right).\end{equation*} 
We compute:
\begin{align*}
\vert \varphi_c(y)\vert - \frac{c}{c^\star}\vert y\vert = \frac{\eta}{2} \left(1-\left(\frac{c^\star}{c}\right)^2\right)\frac{\vert \varphi_c(y)^\perp\vert^2}{\vert \varphi_c(y)\vert}
\end{align*}
 Going back to the PSF, after a change of variable in $\omega$,  the following asymptotic expression when $\eta\rightarrow 0$: 
 
  \begin{multline*}
 F^{c}(z,y) = \eta^{-1} \left(\frac{c^\star}{c}\right)^2 \frac{\ell_0^4}{\left(16\pi^2|y||\varphi_c^{-1}(z)| \right)^2}    \int_{\mathcal{B}_0}\omega^2   e^{i\frac{2\omega}{\eta c^\star}\left( |\varphi_c^{-1}(z)| - |y| \right)}  e^{i\frac{\omega}{c^\star}\left(\left(\frac{c}{c^\star}\right)^2-1\right) \frac{\vert \varphi_c^{-1}(z)^\perp\vert^2}{\vert \varphi^{-1}_c(z)\vert}}
\\ \mathcal{G}^2\left(\frac{\omega\ell_0}{c^\star}\left( \frac{\varphi_c^{-1}(z)^\perp}{|\varphi_c^{-1}(z)|}- \frac{y^\perp}{|y|}\right) , \frac{\omega\ell_0^2}{c^\star}\left( \left(\frac{c^\star}{c}\right)^2 \frac{1}{|\varphi_c^{-1}(z)|} - \frac{1}{|y|}\right)\right)\dd \omega+ \mathcal{O}(1).
 \end{multline*}
%
 \end{proof}

%
%

\begin{corollary}\label{cor:DLF} For \mbox{$z= \varphi_c(y_0) + (\eta^{\frac{1}{2}} \Delta z^\perp,\eta \Delta z^\shortparallel)$, } the expression of the point spread function becomes 
  \begin{multline}\label{eq:PSFDelta}
 F^{c}(z ,y_0) = \eta^{-1}\left(\frac{c^\star}{c}\right)^2 \frac{\ell_0^4}{\left(16\pi^2\vert y_0\vert^2 \right)^2}    \int_{\mathcal{B}_0}\omega^2 e^{i\frac{\omega}{\vert y_0\vert c^\star} \vert y_0^\perp\vert^2 \left(  \left(\frac{c}{c^\star}\right)^2-1 \right)}    e^{i\frac{2\omega}{c}  \Delta z^\shortparallel}  e^{i\frac{\omega}{\left\vert y_0\right\vert c^\star}\left(\frac{c^\star}{c}\right)^2|\Delta z^\perp|^2} 
\\ \mathcal{G}^2\left(\frac{\omega\ell_0}{c^\star  \left\vert y_0\right\vert} \left(\frac{c^\star}{c}\right)^2 \Delta z^\perp , \frac{\omega\ell_0^2}{c^\star\left\vert y_0\right\vert} \left( \left(\frac{c^\star}{c}\right)^2-1\right)\right)\dd \omega  + \mathcal{O}(1). 
 \end{multline}

If ${\mathcal{B}_0} <\omega_0$  we are in the so-called  \emph{broadband regime}(see section \ref{sec:geometry} and \cite[chapter 6]{garnier2016passive}) where the expression above simplifies to:
\begin{multline}\label{PSFbroadband}
F^{c^\star}(z,y) = \eta^{-1}\left(\frac{c^\star}{c}\right)^2 \frac{\omega_0^2\ell_0^4}{\left(16\pi^2(\vert y_0\vert)^2 \right)^2} e^{i\frac{2\omega_0}{c}  \Delta z^\shortparallel}   e^{i\frac{\omega_0}{\vert y_0\vert c^\star} \vert y_0^\perp\vert^2 \left(  \left(\frac{c}{c^\star}\right)^2-1 \right)}  e^{i\frac{\omega_0}{\left\vert y_0\right\vert c^\star}\left(\frac{c^\star}{c}\right)^2|\Delta z^\perp|^2}   \\ \mathrm{sinc\, } \left(\frac{B}{c^\star}\left(\frac{c^\star}{c}\Delta z^\shortparallel+  \left(\frac{c^\star}{c}\right)^2\frac{1}{\vert y_0\vert }|\Delta z^\perp|^2 +\left( \left(\frac{c}{c^\star}\right)^2-1\right)\frac{1}{\vert y_0\vert }|y_0^\perp|^2 \right)\right)
\\ \mathcal{G}^2\left(\frac{\omega_0\ell_0}{c^\star  \left\vert y_0\right\vert} \left(\frac{c^\star}{c}\right)^2 \Delta z^\perp , \frac{\omega_0\ell_0^2}{c^\star\left\vert y_0\right\vert} \left( \left(\frac{c^\star}{c}\right)^2-1\right)\right)\dd \omega  + \mathcal{O}(1) .
\end{multline}
\end{corollary}

\begin{proof}

\textcolor{black}{
Let $z:=\varphi_c(y_0)+ \Delta z$ with $\Delta z :=(\eta^\frac{1}{2}\Delta z^\perp, \eta \Delta z^\shortparallel)$.}
We can write:
\textcolor{black}{
\begin{equation*}
|z|=|\varphi_c(y_0)|\left(1+\frac{\eta}{2}\frac{|\Delta z^\perp|^2}{|\varphi_c(y_0)|^2}+\eta\frac{\Delta z^\shortparallel}{|\varphi_c(y_0)|}\right)+\mathcal{O}(\eta^{\frac{3}{2}}).
\end{equation*}
}
and then:
\textcolor{black}{
\begin{equation*}\frac{c^\star}{c}|z|-|y_0|=\frac{c^\star}{c}|\varphi_c(y_0)|-|y_0|+\frac{\eta}{2}\frac{c^\star}{c}\frac{|\Delta z^\perp|^2}{|\varphi_c(y_0)|}+\eta\frac{c^\star}{c}\Delta z^\shortparallel+\mathcal{O}(\eta^{\frac{3}{2}}).
\end{equation*}}
Using \eqref{eq:defphi} we get
\begin{equation*}\frac{c^\star}{c}|z|-|y_0|=\frac{\eta}{2}\frac{\vert y_0^\perp\vert^2}{\vert y_0\vert} \left(  \left(\frac{c}{c^\star}\right)^2-1 \right)+\frac{\eta}{2}\left(\frac{c^\star}{c}\right)^2\frac{|\Delta z^\perp|^2}{|y_0|}+\eta\frac{c^\star}{c}\Delta z^\shortparallel+\mathcal{O}(\eta^{\frac{3}{2}}).
\end{equation*}
 Similarly
 \begin{equation*}\frac{c^\star}{c}\frac{1}{|z|}-\frac{1}{|y_0|}=\frac{1}{|y_0|}\left(\left(\frac{c^\star}{c}\right)^2-1\right)+\mathcal{O}(\eta^{\frac{1}{2}}).
\end{equation*}
We can also write

 \begin{equation*}\begin{array}{ll}\vsd\dst\frac{c^\star}{c}\frac{\langle z^\perp,x_e^\perp\rangle}{|z|}-\frac{\langle y_0^\perp,x_e^\perp\rangle}{|y_0|}&\dst=\frac{1}{|y_0|}\left(\left(\frac{c^\star}{c}\right)^2\langle z^\perp,x_e^\perp\rangle-\langle y_0^\perp,x_e^\perp\rangle\right)+\mathcal{O}(\eta)\\\vsd \dst&\dst=\frac{1}{|y_0|}\left(\left(\frac{c^\star}{c}\right)^2\langle \varphi_c(y_0)^\perp,x_e^\perp\rangle-\langle y_0^\perp,x_e^\perp\rangle\right)\\ & \dst +\frac{1}{|y_0|}\left(\frac{c^\star}{c}\right)^2\langle\Delta z^\perp ,x_e^\perp\rangle  +\mathcal{O}(\eta^\frac{1}{2}).\end{array}
\end{equation*}
Using \eqref{eq:defphi} again we get:
 \begin{equation*}\frac{c^\star}{c}\frac{\langle z^\perp,x_e^\perp\rangle}{|z|}-\frac{\langle y_0^\perp,x_e^\perp\rangle}{|y_0|}=\frac{1}{|y_0|}\left(\frac{c^\star}{c}\right)^2\langle\Delta z^\perp ,x_e^\perp\rangle +\mathcal{O}(\eta^\frac{1}{2}).
\end{equation*}
 Finally we obtain the following asymptotic expansion for the phase term in the imaging domain \emph{i.e. } at point $z=\varphi_c(y_0)+ \Delta z \in D'$,  with $\Delta z :=(\eta^\frac{1}{2}\Delta z^\perp, \eta \Delta z^\shortparallel)$ and $y_0\in D$:
 \begin{multline*}
\frac{c^\star}{c}|x_e-z|-|x_e-y_0|=\frac{\eta}{2}\frac{\vert y_0^\perp\vert^2}{\vert y_0\vert} \left(  \left(\frac{c}{c^\star}\right)^2-1 \right) + \frac{\eta}{2}\left(\frac{c^\star}{c}\right)^2\frac{|\Delta z^\perp|^2}{|y_0|}+\eta\frac{c^\star}{c}\Delta z^\shortparallel\\-\frac{\eta}{|y_0|}\left(\frac{c^\star}{c}\right)^2\langle\Delta z^\perp ,x_e^\perp\rangle+\frac{\eta}{2} \frac{|x_e^\perp|^2}{|y_0|}\left(\left(\frac{c^\star}{c}\right)^2-1\right)+\mathcal{O}(\eta^{\frac{3}{2}}).
 \end{multline*}

%

Plugging in the phase term in the expression for the point spread function we get that \eqref{eq:PSF2} becomes
  \begin{multline*}
 F^{c}(z ,y_0) = \eta^{-1}\left(\frac{c^\star}{c}\right)^2 \frac{\ell_0^4}{\left(16\pi^2(y_0^\shortparallel)^2 \right)^2}    \int_{\mathcal{B}_0}\omega^2 e^{i\frac{\omega}{\vert y_0\vert c^\star} \vert y_0^\perp\vert^2 \left(  \left(\frac{c}{c^\star}\right)^2-1 \right)}  e^{i\frac{2\omega}{c}  \Delta z^\shortparallel}  e^{i\frac{\omega}{\left\vert y_0\right\vert c^\star}\left(\frac{c^\star}{c}\right)^2|\Delta z^\perp|^2} 
\\ \mathcal{G}^2\left(\frac{\omega\ell_0}{c^\star  \left\vert y_0\right\vert} \left(\frac{c^\star}{c}\right)^2 \Delta z^\perp , \frac{\omega\ell_0^2}{c^\star\left\vert y_0\right\vert} \left( \left(\frac{c^\star}{c}\right)^2-1\right)\right)\dd \omega  + \mathcal{O}(1). 
 \end{multline*} 
 
\end{proof}

\begin{figure}[h!]

\centering
%
%
\begin{tikzpicture}[scale=1,every node/.style={scale=1}]

\begin{axis}[%
width=3.36in,
height=2.88in,
at={(0.923in,0.642in)},
scale only axis,
point meta min=0.187692122921525,
point meta max=199.959807269087,
axis on top,
xmin=-10.0334448160535,
xmax=10.0334448160535,
xlabel style={font=\color{white!15!black}},
xlabel={$\xi_1$},
y dir=reverse,
ymin=-10.0334448160535,
ymax=10.0334448160535,
ylabel style={font=\color{white!15!black}},
ylabel={$\xi_2$},
axis background/.style={fill=white},
title style={font=\bfseries},
title={$\left\vert\mathcal{G}(\xi_1,\xi_2)\right\vert$},
colormap={mymap}{[1pt] rgb(0pt)=(0.2422,0.1504,0.6603); rgb(1pt)=(0.2444,0.1534,0.6728); rgb(2pt)=(0.2464,0.1569,0.6847); rgb(3pt)=(0.2484,0.1607,0.6961); rgb(4pt)=(0.2503,0.1648,0.7071); rgb(5pt)=(0.2522,0.1689,0.7179); rgb(6pt)=(0.254,0.1732,0.7286); rgb(7pt)=(0.2558,0.1773,0.7393); rgb(8pt)=(0.2576,0.1814,0.7501); rgb(9pt)=(0.2594,0.1854,0.761); rgb(11pt)=(0.2628,0.1932,0.7828); rgb(12pt)=(0.2645,0.1972,0.7937); rgb(13pt)=(0.2661,0.2011,0.8043); rgb(14pt)=(0.2676,0.2052,0.8148); rgb(15pt)=(0.2691,0.2094,0.8249); rgb(16pt)=(0.2704,0.2138,0.8346); rgb(17pt)=(0.2717,0.2184,0.8439); rgb(18pt)=(0.2729,0.2231,0.8528); rgb(19pt)=(0.274,0.228,0.8612); rgb(20pt)=(0.2749,0.233,0.8692); rgb(21pt)=(0.2758,0.2382,0.8767); rgb(22pt)=(0.2766,0.2435,0.884); rgb(23pt)=(0.2774,0.2489,0.8908); rgb(24pt)=(0.2781,0.2543,0.8973); rgb(25pt)=(0.2788,0.2598,0.9035); rgb(26pt)=(0.2794,0.2653,0.9094); rgb(27pt)=(0.2798,0.2708,0.915); rgb(28pt)=(0.2802,0.2764,0.9204); rgb(29pt)=(0.2806,0.2819,0.9255); rgb(30pt)=(0.2809,0.2875,0.9305); rgb(31pt)=(0.2811,0.293,0.9352); rgb(32pt)=(0.2813,0.2985,0.9397); rgb(33pt)=(0.2814,0.304,0.9441); rgb(34pt)=(0.2814,0.3095,0.9483); rgb(35pt)=(0.2813,0.315,0.9524); rgb(36pt)=(0.2811,0.3204,0.9563); rgb(37pt)=(0.2809,0.3259,0.96); rgb(38pt)=(0.2807,0.3313,0.9636); rgb(39pt)=(0.2803,0.3367,0.967); rgb(40pt)=(0.2798,0.3421,0.9702); rgb(41pt)=(0.2791,0.3475,0.9733); rgb(42pt)=(0.2784,0.3529,0.9763); rgb(43pt)=(0.2776,0.3583,0.9791); rgb(44pt)=(0.2766,0.3638,0.9817); rgb(45pt)=(0.2754,0.3693,0.984); rgb(46pt)=(0.2741,0.3748,0.9862); rgb(47pt)=(0.2726,0.3804,0.9881); rgb(48pt)=(0.271,0.386,0.9898); rgb(49pt)=(0.2691,0.3916,0.9912); rgb(50pt)=(0.267,0.3973,0.9924); rgb(51pt)=(0.2647,0.403,0.9935); rgb(52pt)=(0.2621,0.4088,0.9946); rgb(53pt)=(0.2591,0.4145,0.9955); rgb(54pt)=(0.2556,0.4203,0.9965); rgb(55pt)=(0.2517,0.4261,0.9974); rgb(56pt)=(0.2473,0.4319,0.9983); rgb(57pt)=(0.2424,0.4378,0.9991); rgb(58pt)=(0.2369,0.4437,0.9996); rgb(59pt)=(0.2311,0.4497,0.9995); rgb(60pt)=(0.225,0.4559,0.9985); rgb(61pt)=(0.2189,0.462,0.9968); rgb(62pt)=(0.2128,0.4682,0.9948); rgb(63pt)=(0.2066,0.4743,0.9926); rgb(64pt)=(0.2006,0.4803,0.9906); rgb(65pt)=(0.195,0.4861,0.9887); rgb(66pt)=(0.1903,0.4919,0.9867); rgb(67pt)=(0.1869,0.4975,0.9844); rgb(68pt)=(0.1847,0.503,0.9819); rgb(69pt)=(0.1831,0.5084,0.9793); rgb(70pt)=(0.1818,0.5138,0.9766); rgb(71pt)=(0.1806,0.5191,0.9738); rgb(72pt)=(0.1795,0.5244,0.9709); rgb(73pt)=(0.1785,0.5296,0.9677); rgb(74pt)=(0.1778,0.5349,0.9641); rgb(75pt)=(0.1773,0.5401,0.9602); rgb(76pt)=(0.1768,0.5452,0.956); rgb(77pt)=(0.1764,0.5504,0.9516); rgb(78pt)=(0.1755,0.5554,0.9473); rgb(79pt)=(0.174,0.5605,0.9432); rgb(80pt)=(0.1716,0.5655,0.9393); rgb(81pt)=(0.1686,0.5705,0.9357); rgb(82pt)=(0.1649,0.5755,0.9323); rgb(83pt)=(0.161,0.5805,0.9289); rgb(84pt)=(0.1573,0.5854,0.9254); rgb(85pt)=(0.154,0.5902,0.9218); rgb(86pt)=(0.1513,0.595,0.9182); rgb(87pt)=(0.1492,0.5997,0.9147); rgb(88pt)=(0.1475,0.6043,0.9113); rgb(89pt)=(0.1461,0.6089,0.908); rgb(90pt)=(0.1446,0.6135,0.905); rgb(91pt)=(0.1429,0.618,0.9022); rgb(92pt)=(0.1408,0.6226,0.8998); rgb(93pt)=(0.1383,0.6272,0.8975); rgb(94pt)=(0.1354,0.6317,0.8953); rgb(95pt)=(0.1321,0.6363,0.8932); rgb(96pt)=(0.1288,0.6408,0.891); rgb(97pt)=(0.1253,0.6453,0.8887); rgb(98pt)=(0.1219,0.6497,0.8862); rgb(99pt)=(0.1185,0.6541,0.8834); rgb(100pt)=(0.1152,0.6584,0.8804); rgb(101pt)=(0.1119,0.6627,0.877); rgb(102pt)=(0.1085,0.6669,0.8734); rgb(103pt)=(0.1048,0.671,0.8695); rgb(104pt)=(0.1009,0.675,0.8653); rgb(105pt)=(0.0964,0.6789,0.8609); rgb(106pt)=(0.0914,0.6828,0.8562); rgb(107pt)=(0.0855,0.6865,0.8513); rgb(108pt)=(0.0789,0.6902,0.8462); rgb(109pt)=(0.0713,0.6938,0.8409); rgb(110pt)=(0.0628,0.6972,0.8355); rgb(111pt)=(0.0535,0.7006,0.8299); rgb(112pt)=(0.0433,0.7039,0.8242); rgb(113pt)=(0.0328,0.7071,0.8183); rgb(114pt)=(0.0234,0.7103,0.8124); rgb(115pt)=(0.0155,0.7133,0.8064); rgb(116pt)=(0.0091,0.7163,0.8003); rgb(117pt)=(0.0046,0.7192,0.7941); rgb(118pt)=(0.0019,0.722,0.7878); rgb(119pt)=(0.0009,0.7248,0.7815); rgb(120pt)=(0.0018,0.7275,0.7752); rgb(121pt)=(0.0046,0.7301,0.7688); rgb(122pt)=(0.0094,0.7327,0.7623); rgb(123pt)=(0.0162,0.7352,0.7558); rgb(124pt)=(0.0253,0.7376,0.7492); rgb(125pt)=(0.0369,0.74,0.7426); rgb(126pt)=(0.0504,0.7423,0.7359); rgb(127pt)=(0.0638,0.7446,0.7292); rgb(128pt)=(0.077,0.7468,0.7224); rgb(129pt)=(0.0899,0.7489,0.7156); rgb(130pt)=(0.1023,0.751,0.7088); rgb(131pt)=(0.1141,0.7531,0.7019); rgb(132pt)=(0.1252,0.7552,0.695); rgb(133pt)=(0.1354,0.7572,0.6881); rgb(134pt)=(0.1448,0.7593,0.6812); rgb(135pt)=(0.1532,0.7614,0.6741); rgb(136pt)=(0.1609,0.7635,0.6671); rgb(137pt)=(0.1678,0.7656,0.6599); rgb(138pt)=(0.1741,0.7678,0.6527); rgb(139pt)=(0.1799,0.7699,0.6454); rgb(140pt)=(0.1853,0.7721,0.6379); rgb(141pt)=(0.1905,0.7743,0.6303); rgb(142pt)=(0.1954,0.7765,0.6225); rgb(143pt)=(0.2003,0.7787,0.6146); rgb(144pt)=(0.2061,0.7808,0.6065); rgb(145pt)=(0.2118,0.7828,0.5983); rgb(146pt)=(0.2178,0.7849,0.5899); rgb(147pt)=(0.2244,0.7869,0.5813); rgb(148pt)=(0.2318,0.7887,0.5725); rgb(149pt)=(0.2401,0.7905,0.5636); rgb(150pt)=(0.2491,0.7922,0.5546); rgb(151pt)=(0.2589,0.7937,0.5454); rgb(152pt)=(0.2695,0.7951,0.536); rgb(153pt)=(0.2809,0.7964,0.5266); rgb(154pt)=(0.2929,0.7975,0.517); rgb(155pt)=(0.3052,0.7985,0.5074); rgb(156pt)=(0.3176,0.7994,0.4975); rgb(157pt)=(0.3301,0.8002,0.4876); rgb(158pt)=(0.3424,0.8009,0.4774); rgb(159pt)=(0.3548,0.8016,0.4669); rgb(160pt)=(0.3671,0.8021,0.4563); rgb(161pt)=(0.3795,0.8026,0.4454); rgb(162pt)=(0.3921,0.8029,0.4344); rgb(163pt)=(0.405,0.8031,0.4233); rgb(164pt)=(0.4184,0.803,0.4122); rgb(165pt)=(0.4322,0.8028,0.4013); rgb(166pt)=(0.4463,0.8024,0.3904); rgb(167pt)=(0.4608,0.8018,0.3797); rgb(168pt)=(0.4753,0.8011,0.3691); rgb(169pt)=(0.4899,0.8002,0.3586); rgb(170pt)=(0.5044,0.7993,0.348); rgb(171pt)=(0.5187,0.7982,0.3374); rgb(172pt)=(0.5329,0.797,0.3267); rgb(173pt)=(0.547,0.7957,0.3159); rgb(175pt)=(0.5748,0.7929,0.2941); rgb(176pt)=(0.5886,0.7913,0.2833); rgb(177pt)=(0.6024,0.7896,0.2726); rgb(178pt)=(0.6161,0.7878,0.2622); rgb(179pt)=(0.6297,0.7859,0.2521); rgb(180pt)=(0.6433,0.7839,0.2423); rgb(181pt)=(0.6567,0.7818,0.2329); rgb(182pt)=(0.6701,0.7796,0.2239); rgb(183pt)=(0.6833,0.7773,0.2155); rgb(184pt)=(0.6963,0.775,0.2075); rgb(185pt)=(0.7091,0.7727,0.1998); rgb(186pt)=(0.7218,0.7703,0.1924); rgb(187pt)=(0.7344,0.7679,0.1852); rgb(188pt)=(0.7468,0.7654,0.1782); rgb(189pt)=(0.759,0.7629,0.1717); rgb(190pt)=(0.771,0.7604,0.1658); rgb(191pt)=(0.7829,0.7579,0.1608); rgb(192pt)=(0.7945,0.7554,0.157); rgb(193pt)=(0.806,0.7529,0.1546); rgb(194pt)=(0.8172,0.7505,0.1535); rgb(195pt)=(0.8281,0.7481,0.1536); rgb(196pt)=(0.8389,0.7457,0.1546); rgb(197pt)=(0.8495,0.7435,0.1564); rgb(198pt)=(0.86,0.7413,0.1587); rgb(199pt)=(0.8703,0.7392,0.1615); rgb(200pt)=(0.8804,0.7372,0.165); rgb(201pt)=(0.8903,0.7353,0.1695); rgb(202pt)=(0.9,0.7336,0.1749); rgb(203pt)=(0.9093,0.7321,0.1815); rgb(204pt)=(0.9184,0.7308,0.189); rgb(205pt)=(0.9272,0.7298,0.1973); rgb(206pt)=(0.9357,0.729,0.2061); rgb(207pt)=(0.944,0.7285,0.2151); rgb(208pt)=(0.9523,0.7284,0.2237); rgb(209pt)=(0.9606,0.7285,0.2312); rgb(210pt)=(0.9689,0.7292,0.2373); rgb(211pt)=(0.977,0.7304,0.2418); rgb(212pt)=(0.9842,0.733,0.2446); rgb(213pt)=(0.99,0.7365,0.2429); rgb(214pt)=(0.9946,0.7407,0.2394); rgb(215pt)=(0.9966,0.7458,0.2351); rgb(216pt)=(0.9971,0.7513,0.2309); rgb(217pt)=(0.9972,0.7569,0.2267); rgb(218pt)=(0.9971,0.7626,0.2224); rgb(219pt)=(0.9969,0.7683,0.2181); rgb(220pt)=(0.9966,0.774,0.2138); rgb(221pt)=(0.9962,0.7798,0.2095); rgb(222pt)=(0.9957,0.7856,0.2053); rgb(223pt)=(0.9949,0.7915,0.2012); rgb(224pt)=(0.9938,0.7974,0.1974); rgb(225pt)=(0.9923,0.8034,0.1939); rgb(226pt)=(0.9906,0.8095,0.1906); rgb(227pt)=(0.9885,0.8156,0.1875); rgb(228pt)=(0.9861,0.8218,0.1846); rgb(229pt)=(0.9835,0.828,0.1817); rgb(230pt)=(0.9807,0.8342,0.1787); rgb(231pt)=(0.9778,0.8404,0.1757); rgb(232pt)=(0.9748,0.8467,0.1726); rgb(233pt)=(0.972,0.8529,0.1695); rgb(234pt)=(0.9694,0.8591,0.1665); rgb(235pt)=(0.9671,0.8654,0.1636); rgb(236pt)=(0.9651,0.8716,0.1608); rgb(237pt)=(0.9634,0.8778,0.1582); rgb(238pt)=(0.9619,0.884,0.1557); rgb(239pt)=(0.9608,0.8902,0.1532); rgb(240pt)=(0.9601,0.8963,0.1507); rgb(241pt)=(0.9596,0.9023,0.148); rgb(242pt)=(0.9595,0.9084,0.145); rgb(243pt)=(0.9597,0.9143,0.1418); rgb(244pt)=(0.9601,0.9203,0.1382); rgb(245pt)=(0.9608,0.9262,0.1344); rgb(246pt)=(0.9618,0.932,0.1304); rgb(247pt)=(0.9629,0.9379,0.1261); rgb(248pt)=(0.9642,0.9437,0.1216); rgb(249pt)=(0.9657,0.9494,0.1168); rgb(250pt)=(0.9674,0.9552,0.1116); rgb(251pt)=(0.9692,0.9609,0.1061); rgb(252pt)=(0.9711,0.9667,0.1001); rgb(253pt)=(0.973,0.9724,0.0938); rgb(254pt)=(0.9749,0.9782,0.0872); rgb(255pt)=(0.9769,0.9839,0.0805)},
colorbar
]
\addplot [forget plot] graphics [xmin=-10.0334448160535, xmax=10.0334448160535, ymin=-10.0334448160535, ymax=10.0334448160535] {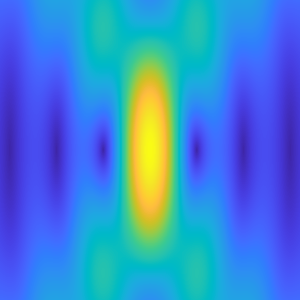};
\end{axis}
\end{tikzpicture}%
\caption{ Plot of $(\xi_1,\xi_2) \mapsto \vert \mathcal{G}(\xi_1,\xi_2)\vert$. }

\end{figure}

\subsection{Discussion} \label{sec:discussionpsf}
We discuss here \eqref{eq:PSFDelta} and \eqref{eq:Fcbroadband}.
For $c=c^\star$ we recover the classical point spread function of a linear array imaging system in the paraxial regime.
In the so-called  \emph{broadband regime} \cite[chapter 6]{garnier2016passive}, for $c=c^\star$, \eqref{eq:Fcbroadband} gives us the classical expected  resolution shown on the middle insert of figure \ref{fig:psf} (see remark \ref{rem:sizefocal}).

The use of an incorrect backpropagation speed $c\neq c^\star$ has three main effects on the focal spot:
\begin{enumerate}
\item The focal spot is not centered around $y_0$ anymore but around the image of $y_0$ by the mapping  between the physical object $D$ and its image $D'$ : \begin{align*}
D &\longrightarrow  D'\\
y_0 & \longmapsto \varphi_c(y_0).
\end{align*}This is a \emph{first order effect} in the sense that the center of the focal spot $\varphi_c(y_0)$ moves proportionnaly to $ \frac{c}{c^\star}-1$ in the axial direction.  
\item The amplitude at the center of the focal spot is decreased.  This can be seen easily by looking at the amplitudes on fig. \ref{fig:psf}.  For a more precise quantification of this effect one can look at the expression \eqref{eq:PSFDelta} at the center of the focal spot (with $z=\varphi_c(y)$) and see that the amplitude is proportional to \begin{align*}
\int_{\mathcal{B}_0} \omega^2 e^{i\omega\frac{|y_0^\perp|^2}{\vert y_0\vert c^\star}\left( \left(\frac{c}{c^\star}\right)^2-1\right)} \mathcal{G}^2\left(0,\frac{\omega \ell_0^2}{c^\star \vert y_0\vert} \left( \left(\frac{c^\star}{c}\right)^2-1\right)\right) \dd \omega.
\end{align*} This is a notable effect as the sensitivity to $\left(\frac{c^\star}{c}\right)^2-1$ is quite high.  With realistic values corresponding to practical situations in medical ultrasound imaging we can see that a $15\%$ error on the speed of sound can lead to an order of magnitude decrease in the amplitude at the center, see fig. \ref{fig:I(z(c))}.  
\item The shape of the focal spot is modified,  it is more spread out and does not have unique local maxima.  However  looking at \eqref{eq:PSFDelta} gives us that $z\mapsto F^c(z,y_0)$ is of order $\eta^{-1}$ in a region of typical size of the order $\eta$ in the axial direction and $\eta^{\frac{1}{2}}$ in the transverse directions which is the same regime as the case $c=c^\star$. The spreading out is due to the shape of $\xi_1\mapsto\mathcal{G}(\xi_1,\xi_2)$ which presents a sharper peak (more contrasted and with a smaller width at half maximum) for $\xi_2$ close to $0$.  Even though this effect is more visually striking,  it is harder to quantify compared to the effect described in the second point.  See \cite{lambert2022}. 
\end{enumerate}
 For a visual representation of these effects we show an illustration on figure \ref{fig:psf} where we plot the point spread function $z\longmapsto  F^{c}(z,y_0)$ in three different cases : $c<c^{\star}, c=c^{\star}$, and $c>c^{\star}$.

\begin{figure}[h!]

\centering
\input{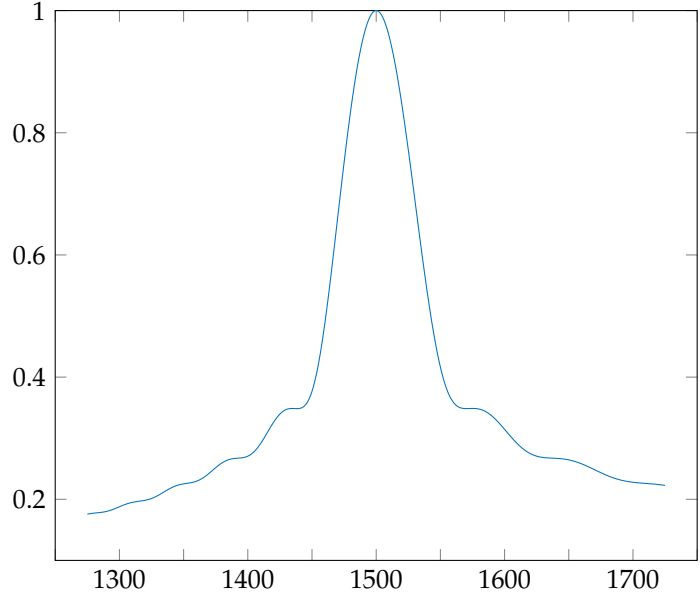}
\caption{ Normalized plot of $\dst c \mapsto \vert\int_\mathcal{B} \omega^2 \mathcal{G}^2\left(0,\frac{\omega \ell^2}{c^\star \vert y_0\vert} \left( \frac{c^\star}{c}\right)^2-1\right) \dd \omega\vert$ for realistic values of the parameters $\ell,  \vert y_0\vert, \mathcal{B},\ldots$ (see section \ref{sec:numerical}).  \label{fig:I(z(c))}}

\end{figure}


\begin{figure}[h!]

\input{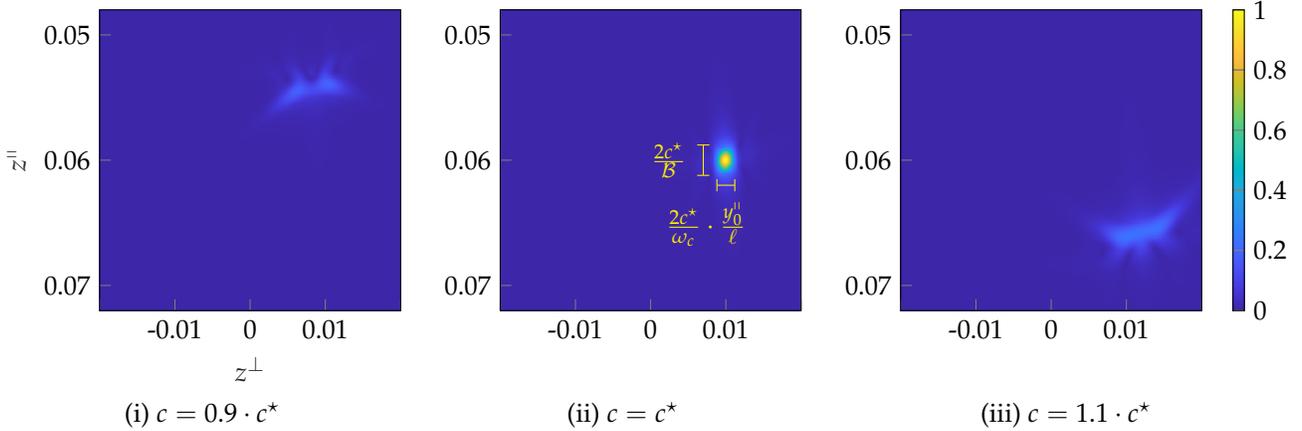} 
\caption{Shape of the point spread function $z\mapsto \left\vert F^c(z,y_0)\right\vert$ with three different backpropagation speed.  Here $y_0=(0.01,0.06)$ for all three graphs.  The quantity $\frac{y_0^\shortparallel}{\ell}$ is the numerical aperture at $y_0$  where $\ell$ is the length of the linear probe $\mathcal{P}$.  The focal spot at threshold $0.1$ is the area in yellow/light blue.  }\label{fig:psf}

\end{figure}

\subsection{Focusing criterion}\label{subsub:focusing}

Consider a point in the medium $y_0\in D$.  Recall that for each $c$ the center of the focal spot in the image is at $\varphi_c(y_0)\in D'$.   Let us define the function
\begin{definition}[Confocal trace at $y_0$]\label{def:confocaltrace}
\begin{align}
c\longmapsto \mathcal{F}_{y_0}(c):= F^c\left( \varphi_c(y_0), y_0\right). 
\end{align}
\end{definition}

\begin{remark} \label{rem:constructionP}Introduced this way, the function $\mathcal{F}_{y_0}$ is an abstract construction but it can be understood more easily if one considers the following thought experiment.  Assume that there is a point-like single isolated reflector located at $y_0$ in an unknown homogeneous medium whose acoustic property $c^\star$ is to be determined. 
In this case the imaging function has a simple approximation:
\begin{align*}
I^c(z)\propto F^c(z,y_0).
\end{align*}
 Up to a horizontal translation of the probe we can assume that this single reflector is located on the axis $y_0^\perp =0$. Of course,  since we do not know the propagation speed inside the medium,  the only thing we can access experimentally is the travel time from the reflector to the center of the probe that we can denote $t_0$ (via the \emph{D.O.R.T} method for example \cite{prada1996decomposition}).  
But that travel time $t_0$ is all one needs to plot the function $c \mapsto I^c(z(c)) $ with $z(c)=(0,c t_0)$.  By noting that $z(c)$ is exactly $\varphi_c(y_0)$ one can see that $c \mapsto I^c(z(c)) $  corresponds  to the value of the pixel at the center of the focal spot for every possible value of $c$,  which is exactly our \emph{confocal trace} $\mathcal{F}$ introduced \emph{i.e. }\begin{align*}
I^c(z(c)) = \mathcal{F}(c).
\end{align*}  The speed of sound $c^\star$ can then be recovered using the estimator introduced in the  proposition \ref{prop:amplitude} that follows.  Once $c^\star$ is recovered,  the actual position $z^\star$ of the reflector can also be recovered. 
\end{remark}
\begin{proposition}[Amplitude characterization]\label{prop:amplitude}
\begin{align*}
c^\star = \mathrm{argmax}_{c} \left\vert \mathcal{F}_{y_0}(c)\right\vert.
\end{align*}
\end{proposition}

\begin{proof}By taking $\Delta z = 0$ in equation \eqref{eq:PSFDelta} we get:
\begin{align*}
\mathcal{F}_{y_0}(c) =\eta^{-1}\left(\frac{c^\star}{c}\right)^2 \frac{\ell_0^4}{\left(16\pi^2(y_0^\shortparallel)^2 \right)^2}    \int_{\mathcal{B}_0}\omega^2   e^{i\frac{\omega}{y_0^\shortparallel c^\star} \vert y_0^\perp\vert^2 \left(\left(\frac{c}{c^\star}\right)^2 -1\right)}  \mathcal{G}^2\left(0, \frac{\omega\ell_0^2}{c^\star y_0^\shortparallel} \left( \left(\frac{c^\star}{c}\right)^2-1\right)\right)\dd \omega  + \mathcal{O}(1).
 \end{align*} 
With similar arguments as in the second point of discussion \ref{sec:discussionpsf} above we recognize a peaked function and see that $c\mapsto\mathcal{F}_{y_0}(c)$ has a maximum at $c=c^\star$. 
 It is easier to see in the so-called \emph{broadband regime} \cite[chapter 6]{garnier2016passive} (${\mathcal{B}_0} <\omega_0$) where the expression above simplifies to:
 \begin{multline}
 \mathcal{F}_{y_0}(c) =\eta^{-1}\left(\frac{c^\star}{c}\right)^2 \frac{\ell_0^4 \omega_0^2  }{\left(16\pi^2(y_0^\shortparallel)^2 \right)^2} \mathrm{sinc}\left( \frac{B}{y_0^\shortparallel c^\star} \vert y_0^\perp\vert^2  \left(\left(\frac{c}{c^\star}\right)^2-1\right)\right)  e^{i\frac{\omega_0}{y_0^\shortparallel c^\star} \vert y_0^\perp\vert^2  \left(\left(\frac{c}{c^\star}\right)^2-1\right) } \\  \mathcal{G}^2\left(0, \frac{\omega_0\ell_0^2}{c^\star y_0^\shortparallel} \left( \left(\frac{c^\star}{c}\right)^2-1\right)\right).\label{eq:Fcbroadband}
 \end{multline}

\end{proof}

\begin{remark}[About the resolution of the estimator]
Looking at \eqref{eq:Fcbroadband} we can see that \mbox{$c\mapsto \mathcal{F}_{y_0}(c)$} has a peak at $c=c^\star$ and that the width at half-maximum in the broadband regime  depends on the numerical aperture at the depth of the focusing point $y_0$, the central frequency $\omega_c$ and is given by : \begin{align*}
\frac{2c^\star y_0^\shortparallel}{\omega_c \ell^2}.
\end{align*}
\end{remark}

\section{Speed of sound estimation in the tissue-mimicking medium}\label{sec:speedofsoundspeckle}

\subsection{The dual focal spot}
With simple symmetry considerations and by exchanging the role of $c$ and $c^\star$ we can see that for a fixed point $z_0\in D'$ the function
 \begin{align*}
D&\longrightarrow D \\ y&\longmapsto F^{c}(z_0,y)
\end{align*}
will have the same qualitative properties as  $ z\longmapsto F^{c}(z,y_0)$. It is then natural to introduce the:
\begin{definition}[Dual focal spot at threshold $\delta$ and speed $c$]
Consider  $0<\delta \ll 1$ and $z_0\in D'$.  Following the same construction as the focal spot,  there exists a domain $\mathcal{D}^c_\delta(\varphi_c^{-1}(z_0))\subset D$  such that \begin{align*}
 \int_{D\setminus \overline{\mathcal{D}^c_\delta(\varphi_c^{-1}(z_0))}}\left\vert F^{c}(z_0,y)\right\vert \dd y <\delta \int_{D}\left\vert   F^{c}(z_0,y)\right\vert \dd y .
\end{align*}
\end{definition}

\begin{remark}
Similarly to the focal spot,  the dual focal spot is not centered at $z_0$ but at \mbox{$\varphi_c^{-1}(z_0)$.}
\end{remark}
\begin{remark}A simple interpretation for the dual focal spot is the following : for a point in the image $z_0\in D'$ and a given back-propagation speed of sound $c$, the dual focal spot $\mathcal{D}^c_\delta(\varphi_c^{-1}(z_0))$ indicates the area in the medium $D$ that mostly contributes to the value of pixel $I^{c}(z_0)$ in the image.  
\end{remark}

\subsection{Locally probing the micro-structure}

In light of the previous definition,  the imaging function depends  on the material's micro-structure  just in the dual focal spot area :\begin{align*}
I^c(z) = &\int_D  \left(n_\varepsilon(y)-n^\star\right) F^c(z,y) \dd y\\
 =& \int_{\mathcal{D}^c_\delta(\varphi_c^{-1}(z_0))}  \left(n_\varepsilon(y)-n^\star\right) F^c(z,y) \dd y + \mathcal{O}(\delta).
\end{align*} This observation is the first step in the construction of the virtual guide star.

The second step is based on the following observation : for a fixed point satisfying the paraxial approximation $z_0=\left(\eta^{\frac{1}{2}}z_0^\perp,z_0^\shortparallel\right)$,  $c\mapsto I^c(z_0)$ probes the  moving area $\mathcal{D}^c_\delta(\varphi_c^{-1}(z_0))$ in the physical domain $D$.  Its center (as a function of $c$) is given by $$c\longmapsto \varphi^{-1}_c(z_0) = \left(\left(\frac{c^\star}{c}\right)^2\eta^{\frac{1}{2}} z_0^\perp,  \frac{c^\star}{c} z_0^\shortparallel \right)$$ and its  size increases with $\vert c-c^\star\vert. $ 
Therefore,  for a fixed $y_0\in D$ the function
\begin{align*}
c \longmapsto I^c\left(\varphi_c(y_0)\right)
\end{align*} probes the area $\mathcal{D}^c_\delta(y_0)$, always centered at $y_0$.   Of course,  its size is still dependent on $\vert c-c^\star\vert$.  

In practice,  $\varphi_c$ is unknown as it depends on $c$ (known) and $c^\star$ (unknown).  However,  as done in \cite{bureau2024reflection}, let us define,  for $s_0\in \R^{d-1}$ and $t_0\in \R^+$:
\begin{align*}
\psi_c(s_0,t_0) := \left( c^2 s_0, c t_0\right).
\end{align*}
Now if we compute \begin{align*}
c\longmapsto I^c\left(\psi_c(s_0,t_0)\right)
\end{align*}
we are in fact computing  $c \mapsto I^c\left(\varphi_c(y_0)\right)$ with an (unknown but fixed) \begin{align*}
y_0 := \left(\eta^{\frac{1}{2}}\left(c^\star\right)^2 s_0, c^\star t_0\right). 
\end{align*}
\begin{remark} For $s_0=0$,  similarly as in remark \ref{rem:constructionP}, $t_0$ corresponds to the travel time from $y_0$ to the center of the probe at speed $c^\star$. 
\end{remark}
From now on,  we choose $s_0\in \R^{d-1}$, $t_0\in \R^+$ and we study the dependency with respect to $c$ of the following functional:
\begin{align}\label{eq:Icpsic}
 I^c\left(\psi_c(s_0,t_0)\right)=& I^c\left(\varphi_c(y_0)\right)\\
 =&\int_{\mathcal{D}^c_\delta(y_0)}  \left(n_\varepsilon(y)-n^\star\right) F^c(\varphi_c(y_0),y) \dd y + \mathcal{O}(\delta).
\end{align}
Since $\mathcal{D}^c_\delta(y_0)$ is a \emph{small} domain centered at $y_0$ we can see that for $y\in \mathcal{D}^c_\delta(y_0)$ we can expect $F^c(\varphi_c(y_0),y)$ to be close to the \emph{confocal trace} $\mathcal{F}_{y_0}(c):=F^c(\varphi_c(y_0),y_0)$.  The difficulty is that $F^c(\varphi_c(y_0),y)$ is integrated against an unknown random, rapidly oscillating -at the scale $\varepsilon r_0 \ll \frac{c^\star}{\omega_c}$ -  index $n_\varepsilon(\cdot)-n^\star$ on the dual focal spot.  
In the next section we show how to overcome this difficulty and build an estimator for the effective speed of sound in the random multi-scale medium.

\begin{proposition}{Imaging in the speckle.}\label{prop:imagingspeckle}
For $s_0 \in \R^{d-1}$, $t_0>0$, let $z(c) := (\eta^{\frac{1}{2}} c^2 s_0, c t_0):=\varphi_c(y_0)$, the expression of the imaging function at $z(c)$ reads
\begin{multline*}
 I^c\left(\varphi_c(y_0)\right)=\eta \left(\frac{c}{c^\star}\right)^2 \frac{\ell_0^4}{\left(16\pi^2|\varphi_c(y_0)|^2 \right)^2} \\  \int_{\Upsilon(\mathcal{D}^c_\delta(y_0))}  \left(n_\varepsilon(y_0+(\tilde{y})_\eta)-n^\star\right)    \int_{\mathcal{B}_0}\omega^2e^{i\frac{2\omega}{c^\star \vert y_0\vert }\left(\frac{c}{c^\star}\right)^2 \left\langle \tilde{y}^\perp,y_0^\perp\right\rangle}e^{i\frac{2\omega}{c^\star} \tilde{y}^\shortparallel} e^{i\frac{\omega}{c^\star|y_0|}\left(\frac{c}{c^\star}\right)^2|\tilde{y}^\perp|^2} \\  e^{i\frac{\omega}{c^\star} \left(\left(\frac{c}{c^\star}\right)^2-1\right) \frac{\vert y_0^\perp +\tilde{y}^\perp\vert^2}{\vert y_0\vert}}\mathcal{G}^2\left(\frac{\omega\ell_0}{c^\star  \left\vert y_0\right\vert} \left(\frac{c}{c^\star}\right)^2\tilde{y}^\perp , \frac{\omega\ell_0^2}{c^\star\left\vert y_0\right\vert} \left( \left(\frac{c^\star}{c}\right)^2-1\right)\right)\dd \omega  \dd \tilde{y} + \mathcal{O}(\eta^2).
 \end{multline*}
 where $(\tilde{y})_\eta:=y-y_0:=(\eta^{\frac{1}{2}}\tilde{y}^\perp, \eta \tilde{y}^\shortparallel)$and $\Upsilon(\mathcal{D}^c_\delta(y_0))$ is the rescaled (by $\eta^{\frac{1}{2}}$ in the transverse directions and $\eta$ in the axial direction) and translated (by $-y_0$) dual focal spot.\vsd

\end{proposition}

\begin{remark}\label{rem:sizedualfocal} By switching the roles of $c$ and $c^\star$ we can see that the focal spot in the imaging domain and the dual focal spot in the physical domain will have similar sizes. The scaling of the dummy variable of integration $(\tilde{y})_\eta$ in proposition \ref{prop:imagingspeckle}   confirms that conjecture and shows that the dual focal spot is indeed of size $\eta$ in  the axial direction and $\eta^{\frac{1}{2}}$ in the transverse direction.  Exactly as the focal spot,  the dual focal spot size and shape are sensitive to the mismatch between $c$ and $c^\star$,  but its axial and transverse sizes stay in the same order of magnitude in the regime considered.
\end{remark}

\begin{proof} 
We know from Lemma \ref{lem:PSFparaxial} that:
\begin{multline*}
 F^{c}(z,y) = \eta^{-1} \left(\frac{c}{c^\star}\right)^2 \frac{\ell_0^4}{\left(16\pi^2|z||\varphi_c(y)| \right)^2}    \int_{\mathcal{B}_0}\omega^2   e^{i\frac{2\omega}{\eta c}\left( |z| - |\varphi_c(y)| \right)} e^{i \frac{\omega}{c}\left( 1-\left(\frac{c^\star}{c}\right)^2\right) \frac{\vert \varphi_c(y)^\perp \vert^2}{\vert \varphi_c(y)\vert}}
\\ \mathcal{G}^2\left(\frac{\omega\ell_0}{c}\left( \frac{z^\perp}{|z|}- \frac{\varphi_c(y)^\perp}{|\varphi_c(y)|}\right) , \frac{\omega\ell_0^2}{c}\left( \frac{1}{|z|} - \left(\frac{c}{c^\star}\right)^2\frac{1}{|\varphi_c(y)|}\right)\right)+ \mathcal{O}(1).
 \end{multline*}
We perform the change of variable $(\tilde{y})_\eta:=y-y_0:=(\eta^{\frac{1}{2}}\tilde{y}^\perp, \eta \tilde{y}^\shortparallel)$.
We have
\begin{equation*}
|\varphi_c(y)|=|\varphi_c(y_0)|+\frac{\eta}{2}\frac{|\varphi_c(\tilde{y})^\perp|^2}{|\varphi_c(y_0)|}+\eta \varphi_c(y)^\shortparallel + \eta \frac{\left\langle \varphi_c(\tilde{y})^\perp,\varphi_c(y_0)^\perp\right\rangle}{\vert \varphi_c(y_0)\vert } +\mathcal{O}(\eta^{\frac{3}{2}}).
\end{equation*}
We get
\begin{equation*}
|z| -|\varphi_c(y)|=\frac{\eta}{2}\frac{|\varphi_c(\tilde{y})^\perp|^2}{|\varphi_c(y_0)|}+\eta \frac{c}{c^\star}\tilde{y}^\shortparallel+ \eta \frac{\left\langle \varphi_c(\tilde{y})^\perp,\varphi_c(y_0)^\perp\right\rangle}{\vert \varphi_c(y_0)\vert } +\mathcal{O}(\eta^{\frac{3}{2}}),
\end{equation*}

\begin{align*}
 \frac{z^\perp}{|z|}- \frac{\varphi_c(y)^\perp}{|\varphi_c(y)|}=&  \frac{\varphi_c(\tilde{y})^\perp}{|\varphi_c(y_0)|}+\mathcal{O}(\eta),
\end{align*} and
\begin{equation*}
\left(\frac{c}{c^\star}\right)^2 \frac{1}{|z|} - \frac{1}{|\varphi_c(y)|} =\frac{1}{|\varphi_c(y_0)|}\left(\left( \frac{c}{c^\star}\right)^2-1\right) +\mathcal{O}(\eta).
\end{equation*}
By plugging those asymptotic expansions back in \eqref{eq:Icpsic} we get the result. 
\end{proof}

\subsection{Effective speed of sound estimation via ensemble averaging}
In thise section we show that it is possible to access the  spatial average of the point-spread function over the dual focal spot via the variance of the imaging function (lemma \ref{lem:incoh}).  Using a similar idea as what was developped in section \ref{subsub:focusing},  we then build an estimator for the effective sound velocity in the medium (proposition \ref{prop:esgtimatorvariance}). 
\begin{lemma}[Second order moment]\label{lem:incoh}
Assume $n:\R^d \times \Omega \rightarrow [n_{-},n_+]$ is a  stationary and mixing process satisfying \eqref{eq:Decaying}.  Then as $\eta\rightarrow 0$ we have:
\begin{align*}
\mathbb{E}\left[\left\vert\int_{\mathcal{D}^c_\delta(\varphi_c^{-1}(z))}  \left(n_\varepsilon(y)-n^\star\right) F^c(z,y) \dd y\right\vert^2 \right]= \varepsilon^d  \int_{\R^d} C(s)\dd s  \int_{\mathcal{D}^c_\delta(\varphi_c^{-1}(z))}   \left\vert F^c(z,y) \right\vert^2\dd y + o\left(\eta^{\frac{d-1}{2}}\right),
\end{align*} where $C$ is the covariance of $n$. 
\end{lemma}
\begin{proof}
By Fubini's theorem:
\begin{align*}
A(z) & :=\mathbb{E}\left[\left\vert\int_{\mathcal{D}^c_\delta(\varphi_c^{-1}(z))}  \left(n_\varepsilon(y)-n^\star\right) F^c(z,y) \dd y\right\vert^2 \right] \\
\vsd & \dst=\iint_{\mathcal{D}^c_\delta(\varphi_c^{-1}(z))^2}\E[\left(n_\eps(y)-n^\star\right)\left(n_\eps(\tilde{y})-n^\star\right)]F^c(z,y)\overline{F^c}(z,\tilde{y})\dd y \dd \tilde{y} \\&\dst=\iint_{\mathcal{D}^c_\delta(\varphi_c^{-1}(z))^2}C\left(\frac{y-\tilde{y}}{\eps}\right)F^c(z,y)\overline{F^c}(z,\tilde{y}))\dd y \dd \tilde{y}
\end{align*}
We change variables to get
\begin{equation*}
A(z) =  \eps^{ d}\iint_{\mathcal{D}^c_\delta(\varphi_c^{-1}(z))\times \mathcal{D}_x} C(x)F^c(z,\tilde{y}+\eps x)\overline{F^c}(z,\tilde{y})\dd \tilde{y} \dd x,
\end{equation*}
where $\mathcal{D}_x:=\dst\left\{\frac{y-\tilde{y}}{\eps}, y,\tilde{y}\in \mathcal{D}^c_\delta(\varphi_c^{-1}(z))\right\}$. 
We want to do a Taylor expansion in \mbox{ $F^c(z, \tilde{y}+\varepsilon x)$} with respect  to $\varepsilon x$ but $\vert x\vert $ can be arbitrarily large at first glance.  However, we can use the separation of scale in our model.  Recall that the micro-structure varies at a scale of order $\varepsilon \ll \eta $ and the point-spread function varies at the scale of the wavelength $c\omega_{\eta}^{-1} =\eta^{-1} \frac{c}{\omega}$.  More precisely, 
since $C$  is controlled by a decreasing function $\Phi$  there exists $R>0$ such that
\begin{equation*}\left|\int_{\R^d}C-\int_{B(0,R) } C\right|\leq \varepsilon.\end{equation*}
Moreover, the decay rate of $\Phi$ gives us $\varepsilon R=o(\eta)$, see lemma \ref{lem:appendix}.
After a Taylor's expansion of $F^c$ at $\tilde{y}$ we obtain,  at first order in $\eta$:
\begin{equation*}
A(z)= \eps^{d}\int_{\R^d} C(s)\dd s \int_{\mathcal{D}^c_\delta(\varphi_c^{-1}(z))}|F^c(z,\tilde{y})|^2\dd \tilde{y} +o\left(\eta^{\frac{d-1}{2}}\right).
\end{equation*}

\end{proof}

We are now ready to state our main proposition:
\begin{proposition}[Effective speed of sound estimator] \label{prop:esgtimatorvariance}
For any $s_0\in \R^{d-1}$,  $t_0\in \R^+$:
\begin{align*}
c^\star = \underset{c}{\mathrm{argmax\, }} \mathbb{E}\left[\left\vert I^c(\psi_c(s_0,t_0))\right\vert^2\right].
\end{align*}
\end{proposition}
\begin{proof}
The typical sizes of $\mathcal{D}_\delta^c(\varphi_c^{-1}(z))$ are $ \eta^{\frac{1}{2}} \times \eta^{-1}$ in dimension $2$ and $ \eta^{\frac{1}{2}} \times \eta^{\frac{1}{2}}\times \eta^{-1}$  in dimension $3$. Therefore  $\int_{\mathcal{D}_\delta^c(\varphi_c^{-1}(z))}  \left\vert F^c(z,y)\right\vert^2 \dd y $ can be reasonably approximated via a Taylor expansion by $  \vert D_\delta^c(\varphi_c^{-1}(z))\vert \left\vert F^c(z,\varphi_c^{-1}(z))\right\vert^2$ and:
\begin{align*}
\mathbb{E}\left[\vert I^c(z,\cdot)\vert^2\right]= \varepsilon^d  \int_{\R^d} C(s)\dd s\,    \vert\mathcal{D}_\delta^c(\varphi_c^{-1}(z))\vert \left\vert F^c(z,\varphi_c^{-1}(z))\right\vert^2 +o\left(\eta^{\frac{d-1}{2}}\right).
\end{align*}
Using $z=\psi_c(s_0,t_0)$ gives us $\varphi_c^{-1}(z)= y_0$ with
 \begin{align*}
y_0 := \left(\eta^{\frac{1}{2}}\left(c^\star\right)^2 s_0, c^\star t_0\right) ,
\end{align*}
which does not depend on $c$ and therefore
\begin{align*}\mathbb{E}\left[\vert I^c(z,\cdot)\vert^2\right]= \varepsilon^d  \int_{\R^d} C(s)\dd s\,  \vert \mathcal{D}_\delta^c(y_0)\vert \left\vert \mathcal{F}_{y_0}(c)\right\vert^2 +o\left(\eta^{\frac{d-1}{2}}\right).
\end{align*}
In light of remark \ref{rem:sizedualfocal} that $\mathbb{E}\left[\vert I^c(z,\cdot)\vert^2\right]$ therefore behaves qualitatively as $\vert \mathcal{F}_{y_0}(c)\vert^2$.
\end{proof}

\subsection{From ensemble averaging to spatial averaging}\label{sec:ensembletospatial}

In the previous subsection we constructed an estimator for the effective speed of sound in the micro-structured random medium. The computation of the estimator requires access to multiple realizations of the imaging functional to approximate its expected value.  However, in practice,  one has access to only  a single realization of the medium.  In this subsection,  we show that we can approximate the expected value with data collected from only one realization of the medium. 

In what follows we  precise the dependency with respect to the realization in the notations.  For a given realization $\varpi \in \Omega$ we write $n(x,\varpi),  x\in D$ and $I^c(z,\varpi), z\in D'$. 
We start by proving a technical lemma.  
In the case where $c=c^\star$ the point spread function $F^{c^\star} (z,y)$ is actually a function of $z-y$.  So it is obvious that for any $\Delta z\in \R^d$ we have $F^{c^\star}(z+\Delta z ,y)= F^{c^\star}(z ,y-\Delta z)$.  In the case $c\neq c^\star$ we prove here a modified version of this result for a \emph{small } (of order $\eta$) $\Delta z$. 
\begin{lemma}\label{lem:symF} For $y_0\in D$ and $y\in\mathcal{D}_\delta^c(y_0+ \varphi_c^{-1}((\Delta z)_\eta))$ with \mbox{$(\Delta z)_\eta := (\eta \Delta z^\perp, \eta \Delta z^\shortparallel)  \in\R^{d}$} \begin{align*}
 F^{c}\left(\varphi_c(y_0)+ (\Delta z)_\eta ,y\right)  = F^c\left(\varphi_c(y_0), y-\varphi_c^{-1}((\Delta z)_\eta)\right) +\mathcal{O}\left(\eta^{-\frac{1}{2}}\right).
 \end{align*}
  \end{lemma}

\begin{proof}
Let us prove the result in dimension $d=3$. 
We know from Lemma  \ref{lem:PSFparaxial}
 \begin{multline*}
 F^{c}(z,y) = \eta^{-1} \left(\frac{c^\star}{c}\right)^2 \frac{\ell_0^4}{\left(16\pi^2|y||\varphi_c^{-1}(z)| \right)^2}    \int_{\mathcal{B}_0}\omega^2   e^{i\frac{2\omega}{\eta c^\star}\left( |\varphi_c^{-1}(z)| - |y| \right)}  e^{i\frac{\omega}{c^\star}\left(\left(\frac{c}{c^\star}\right)^2-1\right) \frac{\vert \varphi_c^{-1}(z)^\perp\vert^2}{\vert \varphi^{-1}_c(z)\vert}}
\\ \mathcal{G}^2\left(\frac{\omega\ell_0}{c^\star}\left( \frac{\varphi_c^{-1}(z)^\perp}{|\varphi_c^{-1}(z)|}- \frac{y^\perp}{|y|}\right) , \frac{\omega\ell_0^2}{c^\star}\left( \left(\frac{c^\star}{c}\right)^2 \frac{1}{|\varphi_c^{-1}(z)|} - \frac{1}{|y|}\right)\right) \dd \omega+ \mathcal{O}(1).
 \end{multline*}
We have, for $y\in \mathcal{D}_\delta\left(y_0\right)$, in light of remark \ref{rem:sizedualfocal} $ (y-y_0)^\shortparallel= \mathcal{O}(\eta)$ and $\vert (y-y_0)^\perp\vert  = \mathcal{O}(1)$. This stays true for $y\in \mathcal{D}_\delta\left(y_0-\varphi_c^{-1}((\Delta z)_\eta)\right)$. First,  we write,  for $z=\varphi_c(y_0)+(\Delta z)_\eta$:
\begin{equation*}
|\varphi_c^{-1}(z)| -|y|=|y_0|-|y|+\eta\frac{c^\star}{c}\Delta z^\shortparallel+\mathcal{O}(\eta^{\frac{3}{2}}).
\end{equation*}
Moreover
\begin{align*}
|y-\varphi^{-1}_c(\Delta z)|=|y|-\eta\frac{c^\star}{c}\Delta z^\shortparallel+\mathcal{O}(\eta^{\frac{3}{2}}).
\end{align*}
We can then write:
\begin{equation}\label{eq:preuve1}
|\varphi_c^{-1}(z)| -|y|=|y_0|-|y-\varphi^{-1}_c(\Delta z)|+\mathcal{O}(\eta^{\frac{3}{2}}).
\end{equation}
Similarly:
\begin{align*}
\varphi_c^{-1}(z)^\perp = y_0^\perp + \mathcal{O}(\eta^{\frac{1}{2}}),
\end{align*}
and
\begin{align}\label{eq:preuve2}\frac{\vert \varphi_c^{-1}(z)^\perp\vert^2}{\vert \varphi_c^{-1}(z)\vert } =\frac{\vert y_0^\perp\vert^2}{\vert y_0\vert} +\mathcal{O}(\eta^{\frac{1}{2}}).
\end{align}
Secondly,
\begin{equation}\label{eq:preuve3}
\begin{aligned}
 \frac{\varphi_c^{-1}(z)^\perp}{|\varphi_c^{-1}(z)|}- \frac{y^\perp}{|y|}=&  \frac{y_0^\perp + \varphi_c^{-1}(\Delta z)^\perp}{|y_0|}-\frac{y^\perp}{|y_0+ (y-y_0)|} +\mathcal{O}(\eta)
 \\=& \frac{y_0^\perp}{|y_0|}-\frac{y^\perp- \varphi_c^{-1}(\Delta z)^\perp }{|y_0|}+\mathcal{O}(\eta)\\=&\frac{y_0^\perp}{|y_0|}-\frac{y^\perp- \varphi_c^{-1}(\Delta z)^\perp }{|y-\varphi_c^{-1}(\Delta z)|}+\mathcal{O}(\eta),
 \end{aligned}
\end{equation} and finally:
\begin{equation}\label{eq:preuve4}
\begin{aligned}
\left(\frac{c^\star}{c}\right)^2 \frac{1}{|\varphi_c^{-1}(z)|} - \frac{1}{|y|} &=\left( \frac{c^\star}{c}\right)^2\frac{1}{|y_0|}-\frac{1}{|y|} +\mathcal{O}(\eta)\\&=\left( \frac{c^\star}{c}\right)^2\frac{1}{|y_0|}-\frac{1}{|y-\varphi_c^{-1}(\Delta z)|} +\mathcal{O}(\eta).
 \end{aligned}
\end{equation}
Plugging back \eqref{eq:preuve1}, \eqref{eq:preuve2},  \eqref{eq:preuve3},  \eqref{eq:preuve4} in the expression of $F$ we recognize $F^c(\varphi_c(y_0),y-\varphi_c^{-1}((\Delta z)_\eta)$.
\end{proof}
 
 Relying on the previous result, we are now ready to state one of the key result of the paper. We prove that, asymptotically,  the imaging function at a point close to $z(c)$ can be written as another realization of the imaging function at the point $z(c)$.
\begin{proposition}{Local stationarity.} \label{prop:localstationarity}
 For $s_0 \in \R^{d-1}$, $t_0>0$, let $z(c) := (\eta^{\frac{1}{2}} c^2 s_0, c t_0)$.  
Then for $a.e.\, \varpi\in\Omega$ and $(\Delta z)_\eta := (\eta \Delta z^\perp,\eta \Delta z^\shortparallel)  \in\R^{d}$,
\begin{align*}
I^c\left(z(c)+(\Delta z)_\eta,\varpi \right) = I^c\left(z(c),\tau_{\varphi_c^{-1}((\Delta z)_\eta)}\varpi\right)+\mathcal{O}(\eta^{1+\frac{d-2}{2}}).
\end{align*}
\end{proposition}
\begin{proof}
The proof is based on a simple change of variable:
\begin{align*}
I^c(z(c)+(\Delta z)_\eta, \varpi) = \int_{\mathcal{D}_\delta^c(\varphi_c^{-1}(z(c)+(\Delta z)_\eta))} \left(n_\eps(y,\varpi)-n^\star\right) F^c(z+\Delta z, y)\dd y
\end{align*}
We recall that since $z(c) =  (\eta^{\frac{1}{2}} c^2 s_0, c t_0)$ we can write $z(c)=\varphi_c(y_0)$ for some $y_0\in D$.  
Since
\begin{align*}
z+(\Delta z)_\eta = \varphi_c (y_0)+(\Delta z)_\eta,
\end{align*}  we know from Lemma \ref{lem:symF}
\begin{align*}
 F^{c}\left(z+ (\Delta z)_\eta ,y\right)  = F^c\left(z, y-\varphi_c^{-1}((\Delta z)_\eta)\right) +\mathcal{O}(\eta^{-\frac{1}{2}})
 \end{align*}
Introducing $\tilde{y}:=y-\varphi_c^{-1}((\Delta z)_\eta)$ and using the stationarity of $n$, we get
\begin{align*}
I^c(z(c)+(\Delta z)_\eta, \varpi) = \int_{\mathcal{D}_\delta^c(y_0)}\left(n_\eps(\tilde{y},\tau_{\varphi_c^{-1}((\Delta z)_\eta)}\varpi)-n^\star\right) F^c(z, \tilde{y})\dd \tilde{y}.
\end{align*}

 \end{proof}
 
We can now use proposition \ref{prop:localstationarity} to show that if one averages spatially the squared modulus of the imaging function over a small area of typical size $a$ of the order of the central wavelength,  then we can approximate the statistical variance of the imaging functional.
\begin{proposition}[From ensemble to spatial averaging] \label{prop:quantificationvariance}For $a\sim\eta \frac{c^\star}{\omega_0}$,
\begin{align*}
\E\left[\left|\E\left[\left\vert I^c(z(c)\right\vert^2 \right]-\fint_{\square{a}} \left|I^c(z(c)+\Delta z)\right|^2\dd \Delta z \right|^2\right]\lesssim \eta^{2\alpha d} \|\Phi^{\frac{1}{2}}(\vert \cdot\vert) \|^2_{L^1(\R^d)} \left(\int_{\mathcal{D}_\delta^c(y_0)}|F^c(z,y)|^2\dd y\right)^2,
\end{align*} where $\square{a}:= [-\frac{a}{2},\frac{a}{2}]^d$. 
\end{proposition}

\begin{proof} We first write
\begin{align*}
\E\left[\left[\fint_{\square{a}}\left|I^c(z(c)+\Delta z)\right|^2\dd \Delta z\right|^2 \right]=\fint_{\square{a}^2}\E\left[\left|I^c(z(c)+\Delta z)\right|^2\left|I^c(z(c)+\Delta z')\right|^2\right]\dd \Delta z\dd \Delta z'
\end{align*}
We know from the previous proposition that
\begin{multline*}
\E\left[\left|I^c(z(c)+\Delta z)\right|^2\left|I^c(z(c)+\Delta z')\right|^2\right]=\int_{\mathcal{D}^c_\delta(y_0)^4}\E[(n_\eps(y-\varphi_c^{-1}(\Delta z))-n^\star)\\(n_\eps(y'-\varphi_c^{-1}(\Delta z))-n^\star)(n_\eps(\tilde{y}-\varphi_c^{-1}(\Delta z'))-n^\star)(n_\eps(\tilde{y}'-\varphi_c^{-1}(\Delta z'))-n^\star)]F^c(z,y)F^c(z,y')F^c(z,\tilde{y})\\F^c(z,\tilde{y}')\dd y \dd y'\dd \tilde{y} \dd \tilde{y}'+\mathcal{O}(\eta^{\frac{3}{2}+2d}).
\end{multline*}
Using \cite[Lemma~IV-4]{kushner1984approximation} or \cite[Lemma~2.1]{bal2008central}, we can bound the $4$th-order moment of $n_\eps$ as follows
\begin{multline}\label{eq:covmult}\left|\E[(n_\eps(y-\varphi_c^{-1}(\Delta z))-n^\star)(n_\eps(y'-\varphi_c^{-1}(\Delta z))-n^\star)\right.\\ \left.(n_\eps(\tilde{y}-\varphi_c^{-1}(\Delta z'))-n^\star)(n_\eps(\tilde{y}'-\varphi_c^{-1}(\Delta z'))-n^\star)]\right|\lesssim \sum_{(y_1,y_2,y_3,y_4) \in \mathcal{U} }\Phi\left(\frac{\vert y_1-y_2\vert }{\eps}\right)^{\frac{1}{2}}\Phi\left(\frac{\vert y_3-y_4\vert }{\eps}\right)^{\frac{1}{2}},\end{multline}
where $\mathcal{U}$ is the set of all permutations of $((y-\varphi_c^{-1}(\Delta z), y'-\varphi_c^{-1}(\Delta z), \tilde{y}-\varphi_c^{-1}(\Delta z'), \tilde{y}'-\varphi_c^{-1}(\Delta z')))$.
After a change of variables, we obtain
\begin{equation*}
\E\left[\left|I^c(z(c)+\Delta z)\right|^2\left|I^c(z(c)+\Delta z')\right|^2\right]\leq \eps^{2 d}\|\Phi^{\frac{1}{2}}\|_{L^1(\R^+)}^2\left(\int_{\mathcal{D}_\delta^c(y_0)}|F^c(z,y)|^2\dd y\right)^2.
\end{equation*}
We used here that $\Phi^{\frac{1}{2}}(\vert \cdot\vert) \in L^1(\R^d)$ (see \eqref{eq:Decaying} and lemma \ref{lem:appendix}).
Similarly we know from Lemma \ref{lem:incoh} that 
\begin{equation*}
\E\left[\left\vert I^c(z(c)\right\vert^2 \right]\sim \eps^{d}\int_{\R^d} C(s)\dd s\,   \int_{\mathcal{D}_\delta^c(y_0)}|F^c(z,y)|^2\dd y.
\end{equation*}
To conclude,  recall that $\varepsilon = \eta^\alpha$ and we get the result.  
\end{proof}

\subsection{Discussion}

\subsubsection{Summary of the results}

\begin{enumerate}
\item For a fixed time $t_0$, for any $c$ the pixel at position $z(c)=(0,ct_0)$ always probes a small area $\mathcal{D}^c(y_0)$  in the domain around $y_0= (0,c^\star t_0)$ regardless of $c$.  
\begin{align*}
I^c(z(c))\sim  \int_{\mathcal{D}^c(y_0)}\left(n_\varepsilon(y)-n^\star\right) F^c(z(c),y) \dd y 
\end{align*}
Proposition \ref{prop:imagingspeckle} gives us an asymptotic expansion of $F^c(z(c),y)$ for $y$ in the small region.    The shape and size of this region $\mathcal{D}^c(y_0)$ depend on $\frac{c^\star}{c}$ via the behavior of $\mathcal{G}_\theta$ but the orders of magnitude of both transverse and axial size of the domain $\mathcal{D}^c(y_0)$ remains the same over a range of $c$ (see remark \ref{rem:sizedualfocal}).  The takeaway is that $y\mapsto F^c(z(c),y)$ varies slowly over $\mathcal{D}^c(y_0)$.  Therefore one can see $I^c(z(c))$ as the convolution of a slowly varying (at the scale of the focal spot) function parametrized by $c$ with the realization of a random process over a domain centered around a fixed point but whose shape is distorted by variations of $c$. 
\item In practice (and in the theoretical model) the separation of scale is such that \mbox{$y\mapsto n_\varepsilon(y)-n^\star$} is the realization of a random process that oscillates spatially at a scale of order $\varepsilon$ much smaller than the size of the dual focal spot $\mathcal{D}^c(y_0)$.    Lemma \ref{lem:incoh} quantifies this qualitative observation and states that the variance of the random variable $I^c(z(c))$ is proportional to $\int_{\mathcal{D}^c(y_0)}\vert F^c(z(c),y)\vert^2 \dd y$. 
\item We know from section \ref{sec:confocal} that $c\mapsto F^c(z(c),y_0)$ is a function strongly peaked at $c=c^\star$ (see figure \ref{fig:I(z(c))}).  Using the asymptotic expansion of proposition  \ref{prop:imagingspeckle} which characterizes the size of the focal spot, we can do a Taylor expansion of $y\mapsto F^c(z(c),y)$ around $y_0$ that is uniformly valid with respect to $c$  in $\mathcal{D}^c(y_0)$.   We then deduce that $\int_{\mathcal{D}^c(y_0)}\vert F^c(z(c),y)\vert^2 \dd y$ behaves like  $\vert F^c(z(c),y_0)\vert^2$ and  exhibits also a strong peak at $c=c^\star$.  Therefore the variance of $c\mapsto I^c(z(c))$ exhibits a maximum at $c=c^\star$. 
\item In practice however,  one has access to only one realization of the medium.  This problem of accessing statistical moments of $c\mapsto I^c(z(c))$ can be worked around using the stationarity of the medium.  Lemma \ref{lem:symF}  states that in the proper asymptotic regime, shifting the imaging point by a \emph{small} $\Delta z$ (of the order of the wavelength) corresponds to  shifting the probed area $\mathcal{D}^c(y_0)$ by some \emph{small} $\Delta y_0$ (of the order of the wavelength too). Since $\Delta y\gg \varepsilon r_0$,  $n_\varepsilon(y)-n^\star$ in $\mathcal{D}^c(y_0)$ and $\mathcal{D}^c(y_0)+\Delta y_0$ can be considered as independent realizations of the distribution of scatterers in  $\mathcal{D}^c(y_0)$ (proposition \ref{prop:localstationarity}). The quantification of the error between the statistical and spatial variance of $c\mapsto I^c(z(c))$ is done in proposition \ref{prop:quantificationvariance}. 
\end{enumerate}

\subsubsection{Practical recovery of $c^\star$ }\label{sec:practical}
We show in this section how the estimator can be constructed in practice as done in \cite{bureau2024reflection} for a two dimensional medium. 
We assume that $c^\star \in [c_{min},c_{max}]$.  We uniformly discretize the set of possible values for $c$ with $N_c$ points  as $c_j= c_{min} + j\Delta c$. 

\begin{enumerate}
\item Fix a travel time $t_0\in \R^+$. The transverse offset $s_0$ can be set to $0$ without loss of generality. 
For every $j\in [1,N_c]$,  consider the moving imaging point
\begin{align*}
z_j=\psi_{c_j}(0,t_0) := \left(0, c_j t_0\right).
\end{align*}
\item Fix a set of $N_\Delta$ spatial shifts $(\Delta z)_i \in \left(\R^2\right)^{N_\Delta}$ satisfying $\varepsilon r_0 \ll \vert \Delta z\vert <\eta \frac{c^\star}{\omega_0}$. 
\item Compute the matrix \begin{align*}
\mathcal{K}_{i,j}= I^{c_j}(z_j+\Delta z_i).
\end{align*}
\item Compute the variance of $K$ with respect to its rows:
\begin{align*}
F_j= \frac{1}{N_{\Delta}} \sum_{i=1}^{N_\Delta} \vert \mathcal{K}_{i,j}^2\vert. 
\end{align*} 
\item Compute $j^\star = \mathrm{argmax}_j \,\vert F_j\vert $.  
\end{enumerate}
We proved that if $N_\Delta$ is large enough then we have 
\begin{align*}
\left \vert c^\star- c_{j^\star} \right\vert \leq \Delta c. 
\end{align*}
\begin{remark}Note that the computation of matrix $M$ at the third step does not require computing the whole image.  Only $N_\delta\times N_c$ pixels of the image are computed.  In practice $N_\Delta$ is of order of a few dozen. 
\end{remark}

\section{Numerical illustrations}\label{sec:numerical}

The numerical simulations are performed in the time domain by using the K-Wave library \cite{treeby2010k}. For a given $x_e \in \R^2$, and a given end time $T > 0$, we simulate $U_e(x, t)$,  the solution in $L^2([0, T], H^1(\R^d \setminus \{{x_e}\}))$ of:
\begin{equation}
\left\{\begin{aligned}
& -\Delta U_e(x, t) + n(x) \partial_{tt}^2 U_e(x, t) = \delta(x - x_e) f(t), && \text{for } x \in \R^2 \setminus \{x_e\},~ t \in [0, T], \\
& U_e(x, 0) = \partial_t U_e(x, 0) = 0,
\end{aligned} \right.
\end{equation}
with $f(t) \in C^0([0, T])$ (see Figure~\ref{fig:InputSignal}). The signal $U_e({x_r}, t)$ is then recorded on the transducers at $x_r$. \\
By considering $N_{e} = 15$ incident waves emitted at $\{x_e\}_{e = 1...N_e}$ and $N_r = 64$ recording sensors placed at $\{x_r\}_{r = 1...N_r}$, we have access to the matrix of data $M(\omega_k) \in \mathbb{C}^{N_{e} \times N_{r}}$ where
$$M(x_e,x_r,\omega_k) = \mathcal{F}(U_e(x_r, \cdot))(\omega_k)$$ 
with $\mathcal{F}$ the Fourier transform and $\{\omega_k\}_{k =1...N_\omega}$ are the sampling frequencies in $\mathcal{B}=[10^2;10^8]$. Here $N_\omega = 1000$. The sensors $\{x_e\}_{e = 1...N_e}$ and $\{x_r\}_{j = 1...N_r}$ are equally spaced on the segment $[-\ell, \ell] \times \{0\} $ with $\ell= 1.5 \times 10^{-2}$ m.  \\
As the problem is a multi-scale problem, the simulation of $M$ can hardly be done on a personal laptop. Indeed, it requires to mesh the small inclusions to capture their effects. With  the CFL condition taken to be $0.4$,  the computations can become quite long. However, for $e=1...N_e$, the simulations of $U_e $ can be done in parallel, and are then computed on Nvidia Tesla V100 GPUs. For the choice of parameters, it takes $\sim 10$ hours per simulation. \\
The parameters $N_e$, $N_r$ and $N_\omega$ are chosen such that the matrix $M$ can be processed on a personal laptop.

We now consider the medium of Figure~\ref{fig:Homogenization} which is composed on 22710 scatterers. The distribution of the centers of the scatterers is a realization of  a Mat\`ern point process (see \cite[Section~6.5.2]{Matern}). In Figure~\ref{fig:Homogenization}, there are $22710$ scatterers with up to $30 \%$ contrast in the speed of sound and a typical radius $\eps r_0 = 7.5 \times 10^{-5}$ m. This corresponds to a particle volume fraction of $15 \%$. 

\begin{figure}
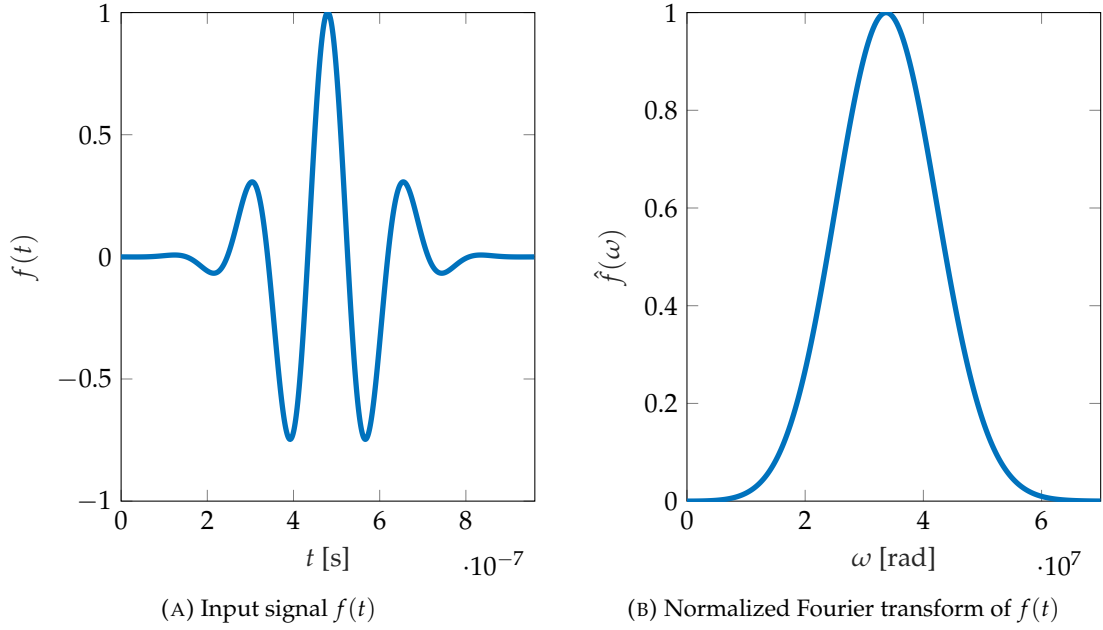

\begin{subfigure}{0.45\textwidth}
\centering
\input{figuresV2/source1.tikz}
\caption{Input signal $f(t)$}
\end{subfigure}
\begin{subfigure}{0.45\textwidth}
\centering
\input{figuresV2/source2.tikz}
\caption{Normalized Fourier transform of $f(t)$}
\end{subfigure}
\caption{Source term used in the simulation.}
\label{fig:InputSignal}
\end{figure}  

\begin{figure}
\centering
\input{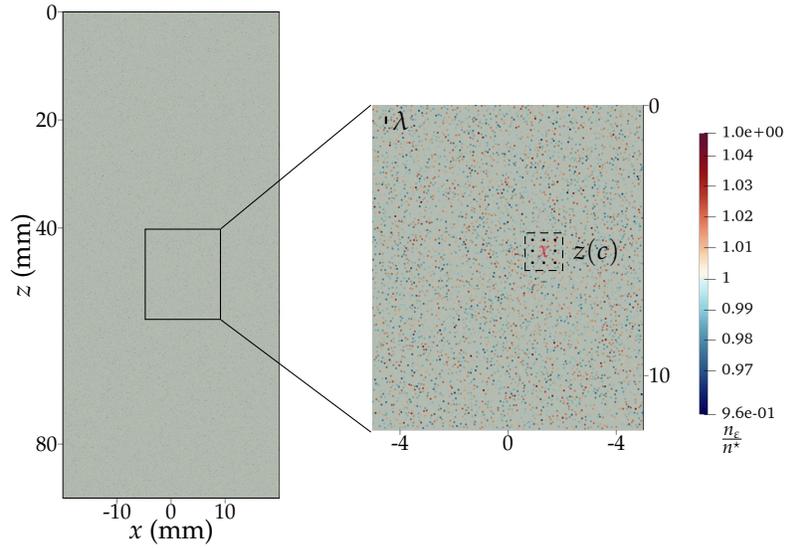}
\caption{Speed of sound map in the random multi-scale medium used in the simulations. }
\label{fig:Homogenization}

\end{figure}

We follow the procedure of section \ref{sec:practical} with $s_0=0$ and  $t_0 = 3\times 10^{-5}$ s. As $c^\star= 1500$ m.s$^{-1}$, this corresponds to $y_0 = (0, 45)$ mm.  The matrix $\mathcal{K}$ is shown on figure \ref{fig:matriceK}. The estimator is plotted on figure \ref{fig:SVDestimators} and  compared to the theoretical expression given in lemma \ref{lem:incoh}.

\begin{figure}

\centering
%
%
\begin{tikzpicture}[scale = 1]

\begin{axis}[%
width= 6.4 cm,
height=7.2cm,
scale only axis,
axis on top,
xlabel style={font=\large\color{white!15!black}},
ylabel style={font=\large\color{white!15!black}},
legend style={legend cell align=left, align=left, draw=white!15!black, font=\tiny},
xtick={0.76, 0.88,...,1.24},
xmin=0.745,
xmax=1.255,
xlabel={$\frac{c_j}{c^\star}$},
y dir=reverse,
ymin=0.5,
ymax=225.5,
ytick={\empty},
ylabel={$(\Delta z)_i$},
axis background/.style={fill=white}
]
\addplot [forget plot] graphics [xmin=0.745, xmax=1.255, ymin=0.5, ymax=225.5] {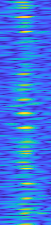};
\end{axis}

\end{tikzpicture}%
\caption{Visualization of  $\vert \mathcal{K}_{ij}\vert = \vert I^{c_j}(z(c_j)+(\Delta z)_i,\varpi_0 )\vert$ for a fixed realization $\varpi_0\in \Omega$. }\label{fig:matriceK}

\end{figure}


\begin{figure}
\center
%
%
\definecolor{mycolor1}{rgb}{0.00000,0.44700,0.74100}%
\definecolor{mycolor2}{rgb}{0.85000,0.32500,0.09800}%
\begin{tikzpicture}[scale=1 ]

\begin{axis}[%
width=3.36in,
height=2.88in,
scale only axis,
xlabel style={font=\normalsize\color{white!15!black}},
ylabel style={font=\normalsize\color{white!15!black}},
axis background/.style={fill=white},
legend style={nodes={scale=1, transform shape},legend cell align=left, align=left, draw=white!15!black},
xmin=0.75,
xmax=1.25,
ymin=0,
ymax=1,
xlabel={$\frac{c}{c^\star}$},
]
\addplot [color=mycolor1, line width=0.6pt]
  table[row sep=crcr]{%
0.75	0.190908172523314\\
0.76	0.193338403281153\\
0.77	0.196957609916415\\
0.78	0.202040234140363\\
0.79	0.208791276171942\\
0.8	0.217316215193736\\
0.81	0.227607050791854\\
0.82	0.239555433168362\\
0.83	0.253001636066092\\
0.84	0.267822176333486\\
0.85	0.284049172313276\\
0.86	0.3020022100934\\
0.87	0.322401048784547\\
0.88	0.346418590678067\\
0.89	0.37563231103163\\
0.9	0.411842325469838\\
0.91	0.456747052386442\\
0.92	0.511501351085069\\
0.93	0.57622146014933\\
0.94	0.64953669332774\\
0.95	0.728308203773792\\
0.96	0.807629449584025\\
0.97	0.881184894231194\\
0.98	0.941974399879971\\
0.99	0.983322022376479\\
1	1\\
1.01	0.989237346169593\\
1.02	0.951371497312368\\
1.03	0.88995460846259\\
1.04	0.811239698259579\\
1.05	0.723122197673088\\
1.06	0.633759500613463\\
1.07	0.550189043220888\\
1.08	0.47727768609738\\
1.09	0.417249031196794\\
1.1	0.36987032726131\\
1.11	0.333186169822912\\
1.12	0.30452697984727\\
1.13	0.281453441333344\\
1.14	0.262350223126023\\
1.15	0.24653558062135\\
1.16	0.233948899367596\\
1.17	0.224637236491214\\
1.18	0.218318521964673\\
1.19	0.21423029820544\\
1.2	0.21131189739083\\
1.21	0.208592024512837\\
1.22	0.205549425270454\\
1.23	0.202235129275208\\
1.24	0.199081560035924\\
1.25	0.196505385158047\\
};
 \addlegendentry{Simulation}

\addplot [color=mycolor2, line width=0.6pt]
  table[row sep=crcr]{%
0.75	0.083972909473642\\
0.75251256281407	0.0867720272506216\\
0.755025125628141	0.0896873881176061\\
0.757537688442211	0.0927183056842237\\
0.760050251256281	0.0958635575096585\\
0.762562814070352	0.099121374767664\\
0.765075376884422	0.102489437151026\\
0.767587939698492	0.105964873479941\\
0.770100502512563	0.10954426845255\\
0.772613065326633	0.113223675942536\\
0.775125628140704	0.116998639208031\\
0.777638190954774	0.120864218327993\\
0.780150753768844	0.124815025126597\\
0.782663316582915	0.128845265783132\\
0.785175879396985	0.132948791254479\\
0.787688442211055	0.137119155559721\\
0.790201005025126	0.141349681892113\\
0.792713567839196	0.145633536432969\\
0.795226130653266	0.14996380964557\\
0.797738693467337	0.154333604725556\\
0.800251256281407	0.15873613277834\\
0.802763819095477	0.163164814184676\\
0.805276381909548	0.167613385503636\\
0.807788944723618	0.172076011149171\\
0.810301507537688	0.176547398963152\\
0.812814070351759	0.181022918695895\\
0.815326633165829	0.185498722295876\\
0.8178391959799	0.189971864805334\\
0.82035175879397	0.194440424559136\\
0.82286432160804	0.198903621292458\\
0.825376884422111	0.20336193068\\
0.827889447236181	0.207817193757447\\
0.830402010050251	0.212272719616262\\
0.832914572864322	0.216733379717412\\
0.835427135678392	0.221205692139863\\
0.837939698492462	0.225697894067093\\
0.840452261306533	0.230220000821106\\
0.842964824120603	0.234783849779544\\
0.845477386934673	0.239403127558934\\
0.847989949748744	0.244093378916593\\
0.850502512562814	0.24887199591629\\
0.853015075376884	0.253758186018781\\
0.855527638190955	0.258772917898242\\
0.858040201005025	0.263938843949357\\
0.860552763819096	0.26928019863715\\
0.863065326633166	0.274822672051995\\
0.865577889447236	0.280593258264599\\
0.868090452261307	0.286620078328962\\
0.870603015075377	0.292932178053649\\
0.873115577889447	0.299559300951222\\
0.875628140703518	0.306531637079999\\
0.878140703517588	0.313879548808745\\
0.880653266331658	0.321633274860358\\
0.883165829145729	0.329822614321776\\
0.885678391959799	0.33847659264041\\
0.888190954773869	0.347623111958572\\
0.89070351758794	0.357288588462185\\
0.89321608040201	0.367497579734394\\
0.89572864321608	0.378272405403651\\
0.898241206030151	0.389632764655\\
0.900753768844221	0.401595354427681\\
0.903266331658291	0.414173492347232\\
0.905778894472362	0.42737674863121\\
0.908291457286432	0.441210591359995\\
0.910804020100503	0.455676049613656\\
0.913316582914573	0.470769399038344\\
0.915829145728643	0.486481874417715\\
0.918341708542714	0.502799413783028\\
0.920854271356784	0.519702438497303\\
0.923366834170854	0.537165673592037\\
0.925879396984925	0.555158012417944\\
0.928391959798995	0.573642429393466\\
0.930904522613065	0.592575944296127\\
0.933417085427136	0.611909641143219\\
0.935929648241206	0.631588744251458\\
0.938442211055276	0.651552753552602\\
0.940954773869347	0.671735640677245\\
0.943467336683417	0.692066106706296\\
0.945979899497487	0.712467901834547\\
0.948492462311558	0.732860206499116\\
0.951005025125628	0.753158072804772\\
0.953517587939699	0.773272924335824\\
0.956030150753769	0.793113111688975\\
0.958542713567839	0.812584520302618\\
0.96105527638191	0.83159122640513\\
0.96356783919598	0.850036196168056\\
0.96608040201005	0.867822022440215\\
0.968592964824121	0.884851692766331\\
0.971105527638191	0.901029381769649\\
0.973618090452261	0.916261260412607\\
0.976130653266332	0.930456314153515\\
0.978643216080402	0.943527161599979\\
0.981155778894472	0.955390864931009\\
0.983668341708543	0.965969723127602\\
0.986180904522613	0.97519203892363\\
0.988693467336683	0.982992850371547\\
0.991206030150754	0.989314618015577\\
0.993718592964824	0.994107858882553\\
0.996231155778894	0.997331718839258\\
0.998743718592965	0.998954475325566\\
1.00125628140704	0.998953963053431\\
1.00376884422111	0.997317915959612\\
1.00628140703518	0.994044219510047\\
1.00879396984925	0.989141068368869\\
1.01130653266332	0.982627025456621\\
1.01381909547739	0.97453097951948\\
1.01633165829146	0.964891999502155\\
1.01884422110553	0.953759085247414\\
1.0213567839196	0.941190815319729\\
1.02386934673367	0.927254894052532\\
1.02638190954774	0.912027601230311\\
1.02889447236181	0.895593149119666\\
1.03140703517588	0.878042952838316\\
1.03391959798995	0.859474821278413\\
1.03643216080402	0.839992076960992\\
1.03894472361809	0.819702614272695\\
1.04145728643216	0.798717906505779\\
1.04396984924623	0.777151972970173\\
1.0464824120603	0.755120318155862\\
1.04899497487437	0.732738855480646\\
1.05150753768844	0.710122828549669\\
1.05402010050251	0.687385743068562\\
1.05653266331658	0.664638322583839\\
1.05904522613065	0.641987501066937\\
1.06155778894472	0.619535465009683\\
1.06407035175879	0.597378757159872\\
1.06658291457286	0.575607453299575\\
1.06909547738693	0.554304422563021\\
1.07160804020101	0.533544680715549\\
1.07412060301508	0.513394844583177\\
1.07663316582915	0.49391269445055\\
1.07914572864322	0.475146849751776\\
1.08165829145729	0.457136561786012\\
1.08417085427136	0.439911625521059\\
1.08668341708543	0.42349241082947\\
1.0891959798995	0.407890011759893\\
1.09170854271357	0.393106510709639\\
1.09422110552764	0.379135352661675\\
1.09673366834171	0.365961823009094\\
1.09924623115578	0.353563620940891\\
1.10175879396985	0.34191151893211\\
1.10427135678392	0.330970097594657\\
1.10678391959799	0.320698544026314\\
1.10929648241206	0.311051500865804\\
1.11180904522613	0.30197995253942\\
1.1143216080402	0.293432134684587\\
1.11683417085427	0.2853544524687\\
1.11934673366834	0.277692393494717\\
1.12185929648241	0.270391421200842\\
1.12437185929648	0.263397835118397\\
1.12688442211055	0.256659585043353\\
1.12939698492462	0.25012702709202\\
1.13190954773869	0.243753610734969\\
1.13442211055276	0.237496487215938\\
1.13693467336683	0.231317031240947\\
1.1394472361809	0.225181269440496\\
1.14195979899497	0.219060210834685\\
1.14447236180905	0.21293007633519\\
1.14698492462312	0.206772426165257\\
1.14949748743719	0.200574185933833\\
1.15201005025126	0.194327573927055\\
1.15452261306533	0.188029933943882\\
1.1570351758794	0.181683479668307\\
1.15954773869347	0.175294958105426\\
1.16206030150754	0.168875240982305\\
1.16457286432161	0.162438854199894\\
1.16708542713568	0.156003456395558\\
1.16959798994975	0.149589278418043\\
1.17211055276382	0.143218536013262\\
1.17462311557789	0.136914828261003\\
1.17713567839196	0.130702534285488\\
1.17964824120603	0.124606220488363\\
1.1821608040201	0.118650070028292\\
1.18467336683417	0.112857345509674\\
1.18718592964824	0.107249894861817\\
1.18969849246231	0.101847709212264\\
1.19221105527638	0.0966685402107431\\
1.19472361809045	0.0917275827748376\\
1.19723618090452	0.0870372276387584\\
1.19974874371859	0.0826068864293898\\
1.20226130653266	0.0784428903069873\\
1.20477386934673	0.0745484615303402\\
1.2072864321608	0.070923755676335\\
1.20979899497487	0.0675659706989355\\
1.21231155778894	0.0644695175878317\\
1.21482412060302	0.0616262461146416\\
1.21733668341709	0.0590257180630426\\
1.21984924623116	0.0566555194525774\\
1.22236180904523	0.0545016026030439\\
1.2248743718593	0.0525486484606522\\
1.22738693467337	0.0507804394259562\\
1.22989949748744	0.0491802329881501\\
1.23241206030151	0.0477311267757693\\
1.23492462311558	0.0464164061690986\\
1.23743718592965	0.0452198663678333\\
1.23994974874372	0.044126101746537\\
1.24246231155779	0.0431207564331102\\
1.24497487437186	0.0421907312806728\\
1.24748743718593	0.041324343736474\\
1.25	0.0405114385058225\\
};
 \addlegendentry{Theoretical}

\addplot [color=black, dashed, line width=0.6pt, forget plot]
  table[row sep=crcr]{%
1	0\\
1	1\\
};
\end{axis}

\end{tikzpicture}%

\caption{Plot of the  estimator function $\dst c_j\longmapsto F_j= \frac{1}{N_{\Delta}} \sum_{i=1}^{N_\Delta} \vert \mathcal{K}_{i,j}^2\vert. $  \label{fig:SVDestimators}}

\end{figure}
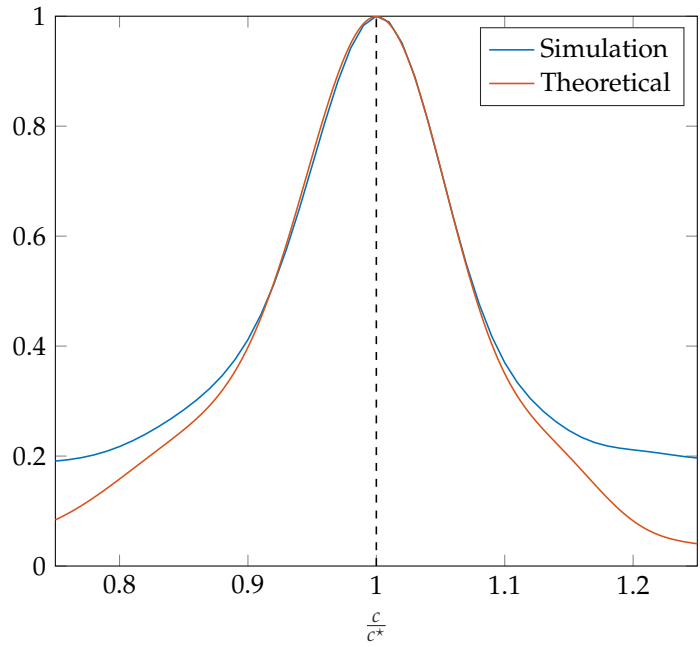 

\section{Perspectives}

We have provided a robust mathematical framework for the analysis of ultrasound imaging signals.   We have proved that that it is possible to recover the speed of sound \emph{in situ} in a homogeneous medium via the construction of sound speed estimators.  Extending these results to more general media (micro-structured with a slowly varying background or a piecewise constant background,  contrast in the divergence\ldots) in order to reconstruct a full map of the speed of sound directly from the backscattering measurements is the subject of a forthcoming work.

\pagebreak

\appendix

\section{Assumptions on the random setting} \label{appendix:random}

Let  $(\Omega, \F, \Pro)$ be a probability space. We define stationarity through an action $(\tau_x)_{x\in\R^d}$ of the group $(\R^d , +)$ on $(\Omega, \F)$ 
that verifies: 
\begin{itemize}
\item[-]the map $\tau: \dst\left\{\begin{array}{ccc}
\vsd\R^d\times \Omega&\longrightarrow&\Omega
\\({x},\varpi)&\longmapsto&\tau_{x} \varpi
\end{array}\right.$ is measurable,
\item[-] $\forall x,y\in\R^d, \tau_{x+y}=\tau_x \circ \tau_y,$ 
\item[-]For all $ x\in\R^d$, $\tau_x$ preserves $\Pro$,  \textit{i.e.}
$$\forall A\in \F,\,\, \Pro(\tau_{x}A)=\Pro(A).$$
\end{itemize} 
We say that a process $f: \R^d \times \Omega \longrightarrow \R$ is stationary if it verifies $a.e.$ $\varpi \in \Omega$, 
\begin{equation} \label{eq:stationary}
f(x + y, \varpi) =f(x, \tau_y \varpi), \qquad \forall x,y \in \R^d.
\end{equation}
We suppose moreover that the action $(\tau_x)_{x\in \R^d}$ is ergodic.

\begin{lemma}[Consequences of \eqref{eq:Decaying}]\label{lem:appendix}
(a) $\Phi(\vert \cdot \vert) \in L^1(\R^d)$ and $\Phi^{\frac{1}{2}}(\vert \cdot \vert) \in L^1(\R^d)$. 
\vsd\\ 
(b) Moreover,  for $R$ large enough \begin{align*}
\left\vert \int_{\R^d}  C(x)\dd x - \int_{B(0,R)} C(x) \dd x \right\vert \lesssim R^{d-1-p}.
\end{align*}
In particular,  for a given $\varepsilon>0$ one can choose $R= \varepsilon^{\frac{1}{d-1-p}}$ and get \begin{align*}
\left\vert \int_{\R^d}  C(x)\dd x - \int_{B(0,R)} C(x) \dd x \right\vert  \lesssim \varepsilon.
\end{align*}
\end{lemma}

\begin{remark}\label{rem:remarkA1}
In practice the decay  rate of $\Phi$ -see eq. \eqref{eq:Decaying}-  encodes the rate of decorrelation of $n$. The idea is that $\Phi$ decays at a scale one order of magnitude smaller than the wavelength,  \emph{i.e} to have $\varepsilon R =o(\eta \frac{c^\star}{\omega_0})$. 
The condition $p>2(d-1)$ is enough for $\Phi(\vert \cdot \vert) \in L^1(\R^d)$ and $\Phi^{\frac{1}{2}}(\vert \cdot \vert)\in L^1(\R^d)$. Moreover this condition is also compatible with the asymptotic regime described in section \ref{sec:asymptotic} in the sense that the condition on $\alpha$ derived  in remark \ref{rem:alpha} is more restrictive.  
\end{remark}

\section{Homogenization in large domains}
\begin{proposition}\label{prop:homolargedomain}
Let $u$ be the solution of \eqref{eq:divformhelmholtz} in $H^2_{loc}(\R^d)$ and $u^i$ the incident wave.  Let $\mathcal{U}^s$ be defined for all $x_e, x_r \in \mathcal{P}$, $\omega \in \mathcal{B}$ by:
$$\mathcal{U}^s(x_e, x_r, \omega) := \omega^2 \int_{D} (n_\eps(x)-n^\star)\Gamma^\frac{\omega}{c^\star}(x, x_e) \Gamma^\frac{\omega}{c^\star}(x, x_r) \dd x.$$ 
It holds:
\begin{equation} \label{eq:homogN}
\norme{(u-u^i -\mathcal{U}^s)(x_e, x_r, \omega)}_{L^2(\Omega)} \lesssim C_\omega \varepsilon^{\frac{d}{p'}}
\end{equation}
for some $C_\omega >0$ independent on $x_e$, $x_r$ and $\eps$ and bounded by $\omega^6 \Vert u^i\Vert_{H^1(D)} + \omega^4$ and $p':=\frac{2+\beta}{1+\beta}$ with $\beta$ chosen such that \begin{align*}
\sup_{x\in D}\, \int_D \left(\Gamma^\frac{\omega}{c^\star}(x,y)\right)^{2+\beta} \dd y <\infty.
\end{align*}
\end{proposition}

\begin{proof}

The starting point is the integral representation formula for the scattered field:
$$u(y)-u^i(y) =\omega^2\int_{D}  \left(n_\eps(x)-n^\star\right) \Gamma^\frac{\omega}{c^\star}(x, y)u(x)\dd x, $$ for $y \in \R^d \setminus \{x_e\}$.
Replacing twice $u$ by its integral representation inside the integral leads to:
\begin{equation}
\begin{split}
u(x_r)-u^i (x_r) & = \int_{D} \omega^2 \left(n_\eps(x)-n^\star\right)  \Gamma^\frac{\omega}{c^\star}(x,x_r) u^i(x) \dd x\\
& \hspace*{-2.5cm} +\omega^4 \int_{D^2} \left(n_\eps(x_1)-n^\star\right)  \left(n_\eps(x_2)-n^\star\right)   \Gamma^\frac{\omega}{c^\star}(x_1, x_2) \Gamma^\frac{\omega}{c^\star}(x_1, x_r) u^i(x_2) \dd x_1 \dd x_2 \\
& \hspace*{-2.5cm} + \omega^6 \int_{D^3}\left(n_\eps(x_1)-n^\star\right)  \left(n_\eps(x_2)-n^\star\right)\left(n_\eps(x_3)-n^\star\right)     \Gamma^\frac{\omega}{c^\star}(x_3, x_2)  \Gamma^\frac{\omega}{c^\star}(x_2, x_1)  \Gamma^\frac{\omega}{c^\star}(x_1, x_r) u(x_3) \dd x_1\dd x_2\dd x_3\\
& :=\mathcal{U}^s(x_r) + \mathcal{U}^{s,2}(x_r)+ \mathcal{U}^{s,3}(x_r). 
\end{split}
\end{equation}

To estimate $\mathcal{U}^{s,2}$ and $\mathcal{U}^{s,3}$ we use  \cite[Lemma~2.1]{bal2008central} with $q=n-n^\star$ (recall that $n_\varepsilon=n(\frac{\cdot}{\varepsilon})$).  The hypothesis $(6)$ from \cite{bal2008central} is verified because the covariance $C$ of $n_\varepsilon$ is  bounded by $\Phi$ in $L^1(\R^+)\cap L^{\frac{1}{2}}(\R^+)$.  
Therefore we have:
\begin{align*}
\left\vert \mathbb{E} \left[ q(x_1)  q(x_2)  q(x_3)  q(x_4) \right]\right\vert \lesssim \sup_{( y_1, y_2,y_3,y_4)\in \Xi} \left(\Phi(\vert y_1-y_2\vert)\right)^{\frac{1}{2}}  \left(\Phi(\vert y_3-y_4\vert)\right)^{\frac{1}{2}} 
\end{align*}
where $\Xi$ is the set of permutations of $(x_1, x_2,x_3, x_4)$.
We have,  following the same steps as in the proof of  \cite[Lemma~2.6]{bal2008central} 
\begin{align*}
\mathbb{E}\left[ \vert \mathcal{U}^{s,2}(x_r)\vert^2\right] \lesssim \omega^8\varepsilon^{\frac{2d}{p'}} \Vert \Gamma^\frac{\omega}{c^\star}(\cdot, x_r) \Vert_{L^\infty(D)}^2 \Vert u^i \Vert_{L^\infty(D)}^2  \Vert\Phi^{\frac{1}{2}}\Vert_{L^{p'}}^2 \Vert \Gamma^\frac{\omega}{c^\star}(\cdot, \cdot) \Vert_{L^p(D^2)}^2,
\end{align*}
with $p=2+\beta$ and $p'=\frac{2+\beta}{1+\beta}$ with $\beta$ chosen such that \begin{align*}
\sup_{x\in D}\, \int_D \left(\Gamma^\frac{\omega}{c^\star}(x,y)\right)^{2+\beta} \dd y <\infty.
\end{align*}
The term $\mathcal{U}^{s,3}$ can be dealt with in a similar manner.  The difference is the presence of the term $u$ instead of $u^i$ in the integrand.
To deal with this term we first recall that a.s. :
\begin{align*}
\Vert u\Vert_{H^1(D)}  \lesssim  \Vert u^i \Vert_{H^1(D)}.
\end{align*}
Then we can perform a Cauchy-Schwarz type estimate to separate the term with $u$. 
\begin{multline*}
\mathbb{E}\left[ \vert \mathcal{U}^{s,3}(x_r)\vert^2\right] \lesssim  \mathbb{E}\left[ \int_D \left\vert \int_{D^2}  \left(n_\eps(x_1)-n^\star\right)  \left(n_\eps(x_2)-n^\star\right)   \Gamma^\frac{\omega}{c^\star}(x_1, x_2) \Gamma^\frac{\omega}{c^\star}(x_1, x_r)\right.  \right. \\ \left. \left.  \Gamma^\frac{\omega}{c^\star}(x_3, x_2) \dd x_1 \dd x_2\right\vert^2 \dd x_3 \int_D  \left(n_\eps(x_1)-n^\star\right)^2 u^2(x_1) \dd x_1 \right].
\end{multline*}
In the end we obtain:
\begin{multline*}
 \mathbb{E}\left[ \vert \mathcal{U}^{s,3}(x_r)\vert^2\right] \lesssim \omega^{12} \varepsilon^{\frac{2d}{p'}} \Vert \Gamma^\frac{\omega}{c^\star}(\cdot, x_r) \Vert_{L^\infty(D)}^2 \Vert u^i \Vert_{H^1(D)}^2  \Vert\Phi^{\frac{1}{2}}\Vert_{L^{p'}}^2 \Vert \Gamma^\frac{\omega}{c^\star}(\cdot, \cdot) \Vert_{L^p(D^2)}^2  \\ \sup_{x\in D}\, \int_D \left(\Gamma^\frac{\omega}{c^\star}(x,y)\right)^{2} \dd y .
\end{multline*}
\end{proof}
 \bibliographystyle{plain}
 \bibliography{biblio}

\end{document}